\documentclass[11pt]{amsart}
\usepackage{oldgerm}
\usepackage{amssymb} 
\usepackage{mathrsfs} 
\usepackage{amsmath}
\usepackage{latexsym}
\usepackage[dvips]{graphics}
\usepackage[dvipdfm,bookmarks=true]{hyperref}
\usepackage{amsthm}
\usepackage{enumerate}

\theoremstyle{plain} 
\newtheorem{thm}{Theorem}[section] 

\newtheorem*{ikedathm}{Theorem I} 
\newtheorem*{hidathm}{Theorem II} 
\newtheorem*{shintanithm}{Theorem III} 
\newtheorem*{stevensthm}{Theorem IV} 
\newtheorem{prop}[thm]{Proposition} 

\newtheorem{lem}[thm]{Lemma} 

\newtheorem{cor}[thm]{Corollary} 

\theoremstyle{definition} 
\newtheorem{defn}[thm]{Definition} 

\newtheorem{ex}[thm]{Example} 

\newtheorem{rem}[thm]{Remark} 

\newtheorem{fact}[thm]{Fact} 
\newtheorem*{acknow}{Acknowledgements} 
\theoremstyle{remark} 
\newtheorem*{notation}{Notation}
\makeatletter
\def\@makefnmark{%
        \leavevmode
        \raise.9ex\hbox{\check@mathfonts
                \fontsize\sf@size\mathbb{Z}@\normalfont%
                        \@thefnmark}%
}
\makeatother

\title[Certain $p$-adic families of Siegel modular forms of even genus]{On certain constructions of $p$-adic families of Siegel modular forms 
of even genus}

\author{Hisa-aki KAWAMURA}
\address{(Primary) Institut Fourier, UFR de Math\'ematiques, Universit\'e Grenoble I, 100 rues des maths, BP 74, 38402 Saint-Martin d'H\`eres cedex, France}
\email{hisa@fourier.ujf-grenoble.fr}
\address{(Secondary) Department of Mathematics, Hokkaido University, Kita 10 Nishi 8, Kita-Ku, Sapporo, 060-0810 Hokkaido, Japan}
\email{kawamura@math.sci.hokudai.ac.jp}

\date{19 September, 2010}                                           

\dedicatory{
{To Professor Hiroshi Saito in memoriam}
}

\thanks{This research was partially supported by the JSPS Core-to-Core Program \lq\lq New Developments of Arithmetic Geometry, Motive, Galois Theory and Their Practical Applications\rq\rq. It is also supported by the JSPS International Training Program (ITP)}

\begin{document}
\maketitle
\begin{abstract}
Suppose that $p > 5$ is a rational prime. Starting from a well-known $p$-adic analytic 
family of ordinary elliptic cusp forms of level $p$ due to Hida, 
we construct a certain $p$-adic analytic family of holomorphic Siegel cusp forms of arbitrary even genus 
and of level $p$ associated with Hida's $p$-adic analytic family via the Duke-Imamo{\={g}}lu lifting provided by Ikeda. 
Moreover, we also give a similar results for the Siegel Eisenstein series of even genus with trivial Nebentypus.  
\end{abstract}

\tableofcontents

\section{Introduction}

\indent 
For a given rational prime $p> 5$, the study of $p$-adic analytic families of modular forms was initiated by Kummer and Eisenstein for the Eisenstein series on the elliptic modular group ${\rm SL}_{2}(\mathbb{Z})$, and afterwards was developed from various points of view by Iwasawa, Kubota-Leopoldt, Serre, Swinnerton-Dyer, Katz, Deligne-Ribet, Hida, Wiles and others. 
In particular, Hida \cite{Hid93} and Wiles \cite{Wil88}, among others, introduced the notion of $\Lambda$-adic modular forms, 
where $\Lambda=\mathbb{Z}_p[[1+p \mathbb{Z}_p]]$ is the Iwasawa algebra, as 
formal $q$-expansions with coefficients in a finite flat $\Lambda$-algebra $R$ 
whose specialization 
at each arithmetic point 
in the $\Lambda$-adic analytic space $\mathfrak{X}(R)={\rm Hom}_{\rm cont}(R,\, \overline{\mathbb{Q}}_p)$ 
gives rise to 
the $q$-expansion of 
an elliptic modular 
form.
In this context, a $p$-adic analytic family of modular forms can be regarded as an infinite collection of modular forms parametrized by varying weights whose components are 
interpolated by a $\Lambda$-adic modular form simultaneously. 

\vspace*{2mm}
In fact, a specific use of Hida's theory 
allows us to construct a $\Lambda$-adic modular form such that every specialization 
gives rise to a $p$-ordinary elliptic Hecke eigenform (i.e. a simultaneous eigenfunction of all Hecke operators such that the eigenvalue of the Atkin-Lehner $U_p$-operator is a $p$-adic unit), which is called the universal ordinary $p$-stabilized newform of tame level 1: 

\begin{fact}[cf. \S 3.1 below]
Let $\Lambda_1=\mathbb{Z}_p[[\mathbb{Z}_p^{\times}]]$ be the completed group ring on $\mathbb{Z}_p^{\times}$ over $\mathbb{Z}_p$, and $\omega=\omega_p$ the Teichm\"uller character. 
There exist a $\Lambda_1$-algebra $R^{\rm ord}$ finite flat over $\Lambda$ 
and a formal $q$-expansion $
{\bf f}_{\rm ord} \in R^{\rm ord}[[q]]$ 
such that 
for each arithmetic point 
$P \in \mathfrak{X}(R^{\rm ord})$ of weight $2k >2
$ with trivial Nebentypus (i.e. $P$ lies over the $\overline{\mathbb{Q}}_p^{\times}$-valued 
continuous character $y \mapsto 
y^{2k}\omega(y)^{2k}$ on $\mathbb{Z}_p^{\times}$), the specialization ${\bf f}_{\rm ord}(P)$ coincides with an ordinary $p$-stabilized newform $f_{2k}^{*}$ of weight $2k$ on 
the congruence subgroup $\Gamma_0(p) \subset {\rm SL}_2(\mathbb{Z})$ of level $p$ with trivial Nebentypus. 
Namely, there exists a $p$-ordinary normalized cuspidal Hecke eigenform $f_{2k}=\sum_{m=1}^{\infty} a_m(f_{2k}) q^m$ of weight $2k$ on 
${\rm SL}_2(\mathbb{Z})$ such that 
\[
f_{2k}^{*}(z) = f_{2k}(z) - \beta_p(f_{2k}) f_{2k}(pz) 
\] 
for all $z \in \mathfrak{H}_1:=\{ z \in \mathbb{C}\, |\, {\rm Im}(z) > 0 \}$, and 
$f_{2k}^{*}$ possesses the $U_p$-eigenvalue $\alpha_p(f_{2k})$, 
where $\alpha_p(f_{2k})$ and $\beta_p(f_{2k})$ denote the unit and non-unit $p$-adic roots of the equation 
\begin{equation}
X^2- a_p(f_{2k})
X+p^{2k-1}=0, 
\end{equation}
respectively. 
\end{fact}

In this setting, we pick an integer $k_0
\ge 6
$ as small as possible to have an arithmetic point 
${P_0 \in \mathfrak{X}(R^{\rm ord})}$ of weight $2k_0$ with trivial Nebentypus 
corresponding to an actual modular form ${\bf f}_{\rm ord}(P_0)=
f_{2k_0}^{*}$, and fix an analytic 
neighborhood $\mathfrak{U}_0$ of $P_0$ in $ \mathfrak{X}(R^{\rm ord})$
on which every coefficient of ${\bf f}_{\rm ord}$ can be regarded as a $p$-adic analytic function. 
Then, by varying the weights $2k$ of arithmetic points $P \in \mathfrak{U}_0$ with trivial Nebentypus, 
we obtain an ordinary $p$-adic analytic family $\{f_{2k}^{*}\}$ parametrized by 
weights $2k \ge 2k_0$ with ${2k \equiv 2k_0 \pmod{(p-1)p^{m-1}}}$ 
for some sufficiently large integer $m \ge 1$. 
Hereinafter, we refer to it as the Hida family of level $p$.  

\vspace*{3mm}
On the other hand, as an affirmative answer to the Duke-Imamo{\={g}}lu conjecture 
concerning on a generalization of the Saito-Kurokawa lifting towards higher genus, 
Ikeda \cite{Ike01} established the following: 

\begin{fact}[cf. \S 2.1 below]
For each integer $n \ge 1$, 
let $f$ be a normalized cuspidal Hecke eigenform of weight $2k$ on ${\rm SL}_2(\mathbb{Z})$ with $k \equiv n \pmod{2}$.  
Then there exists a non-zero cuspidal Hecke eigenform ${\rm Lift}^{(2n)}(f)$ of weight $k+n$ on the Siegel modular group ${\rm Sp}_{4n}(\mathbb{Z}) \subset {\rm GL}_{4n}(\mathbb{Z})$ of genus $2n$ such that 
the standard $L$-function $L(s,\, {\rm Lift}^{(2n)}(f),\,{\rm st})$ is taken of the form  
\[
L(s,\, {\rm Lift}^{(2n)}(f),\,{\rm st})=\zeta(s)\prod_{i=1}^{2n}L(s+k+n-i,\,f), 
\]
where $\zeta(s)$ and $L(s,\,f)$ denote Riemann's zeta function and Hecke's $L$-function associated with $f$, respectively. 
\end{fact} 

When $n=1$, ${\rm Lift}^{(2)}(f)$ coincides with the Saito-Kurokawa lifting of $f$, whose existence was firstly conjectured by Saito and Kurokawa \cite{Kur78}, and afterwards was shown by 
Maa\ss  \,\cite{Maa79}, Andrianov \cite{And79} and Zagier \cite{Zag81}. 
In accordance with the tradition, 
we refer to ${\rm Lift}^{(2n)}(f)$ as the Duke-Imamo{\={g}}lu lifting of $f$ throughout the present article. 
We should mention that such particular objects obtained by means of the 
lifting process from lower genus are, of course, not ``genuine'' Siegel cusp forms of higher genus in the strict sense. Indeed, 
the above functoriality equation yields 
that 
the associated Satake parameter 
$(\psi_0(p),\psi_1(p), \cdots, \psi_{2n}(p)) \in (\overline{\mathbb{Q}}_p^{\times})^{2n+1}$ at $p$ 
is taken as 
\[
\psi_i(p)=\left\{
\begin{array}{ll}
p^{nk-n(n+1)/2} & \textrm{if }i=0, \\[1.5mm]
\alpha_p(f) p^{-k+i} & \textrm{if }1 \le i \le n, \\[2mm]
\beta_p(f) p^{-k+i-n} & \textrm{if }n+1 \le i \le 2n,
\end{array}\right.
\]
uniquely up to the action of the Weyl group $W_{2n} \simeq \mathfrak{S}_{2n} \ltimes \{\pm 1\}^{2n}$,  
where $(\alpha_p(f), \beta_p(f))$ denotes the ordered pair of the roots of the same equation as (1) 
with the inequality of $p$-adic valuations $v_p(\alpha_p(f)) \le v_p(\beta_p(f))$. 
Therefore the famous Ramanujan-Petersson conjecture for $f$ 
implies the fact that 
${\rm Lift}^{(2n)}(f)$ generates a non-tempered cuspidal automorphic representations of the group ${\rm GSp}_{4n}(\mathbb{A}_{\mathbb{Q}})$ of symplectic similitudes, where $\mathbb{A}_{\mathbb{Q}}$ denotes the ring of 
adeles of $\mathbb{Q}$. 
However, it has also been observed 
that some significant properties of ${\rm Lift}^{(2n)}(f)$ can be derived from 
corresponding properties of $f$. 
For instance, 
as will be explained more precisely in the sequel, 
the Fourier expansion of ${\rm Lift}^{(2n)}(f)$ can be written explicitly in terms of those of $f$ and a cuspidal Hecke eigenform of half-integral weight corresponding to $f$ via the Shimura correspondence. 

\vspace*{2.5mm}
Now, let us explain our results. 
Let 
$\{f_{2k}\}$ be the infinite 
collection of $p$-ordinary normalized cuspidal Hecke eigenforms on ${\rm SL}_2(\mathbb{Z})$ 
corresponding to the Hida family $\{f_{2k}^{*}\}$ 
via the ordinary $p$-stabilization. 
The aim of the present article is to construct a $p$-adic analytic family of Siegel cusp forms on the congruence subgroup 
$\Gamma_0(p) \subset {\rm Sp}_{4n}(\mathbb{Z})$ of level $p$ corresponding to 
the collection 
$\{{\rm Lift}^{(2n)}(f_{2k})\}$ under a suitable $p$-stabilization process. 
Namely, our main results are summarized as follows: 

\begin{thm}
For each integer $n \ge 1$, let $k_0$ be a positive integer with $k_0 > n+1$ and $k_0 \equiv n \pmod{2}$, $P_0 \in \mathfrak{X}(R^{\rm ord})$ an arithmetic point of weight $2k_0$ with trivial Nebentypus, 
and $\mathfrak{U}_0$ a fixed analytic neighborhood of $P_0$.  
For each integer $k \ge k_0$ with $k \equiv k_0 \pmod{2}$, 
put 
\begin{eqnarray*}
\Phi_p^{*}(Y)
&:=&
(Y-
\alpha_p(f_{2k})^n
)^{-1}
\prod_{r=1}^{2n}\,\, \prod_{1\le i_1 < \cdots < i_r \le 2n} 
(Y - \psi_0(p) \psi_{i_1\!}(p) \cdots \psi_{i_r\!}(p)), \\
\Psi_p^{*}(Y)&:=&
(Y-\alpha_p(f_{2k})^{n-1} p^{k+n-1})
\prod_{i=1}^{n}(Y-
\alpha_p(f_{2k})^{n-1} \beta_p(f_{2k})p^{2i-2}
),
\end{eqnarray*}
and 
\[
{\rm Lift}^{(2n)}(f_{2k})^{*} := {\Psi_p^{*}(\alpha_p(f_{2k})^n) \over 
\Phi_p^{*}(\alpha_p(f_{2k})^n)
} \cdot {\rm Lift}^{(2n)}(f_{2k})\, |_{k+n}
\Phi_p^{*}(U_{p,0}), 
\]
where $(\psi_0(p),\psi_1(p),\cdots,\psi_{2n}(p))
$ is the Satake parameter of ${\rm Lift}^{(2n)}(f_{2k})$ taken as above, and $U_{p,0}$ is the Hecke operator 
corresponding to the double coset 
\[
I_{4n}\, {\rm diag}(\underbrace{1,\cdots,1}_{2n},\underbrace{p,\cdots,p}_{2n})
\, I_{4n}  
\]  
with respect to 
the standard Iwahori subgroup $I_{4n}$ of ${\rm GSp}_{2g}(\mathbb{Z}_p)$. 
Then we have 
\begin{enumerate}[\upshape (i)]
\item 
${\rm Lift}^{(2n)}(f_{2k})^{*}$ is a cuspidal Hecke eigenform of weight $k+n$ on $\Gamma_0(p) \subset {\rm Sp}_{4n}(\mathbb{Z})$ with trivial Nebentypus such that the eigenvalues agree with those of 
${\rm Lift}^{(2n)}(f_{2k})$ 
for each prime $l \ne p$, 
and we have 
\[
{\rm Lift}^{(2n)}(f_{2k})^{*}\,|_{k+n}\,U_{p,0} = \alpha_p(f_{2k})^n\cdot {\rm Lift}^{(2n)}(f_{2k})^{*}. 
\]

\item 
\vspace*{0.5mm}
Let 
$\sigma : \Lambda_1 \to \Lambda_1$ be the ring homomorphism induced from the group homomorphism $y \mapsto y^2$ on $\mathbb{Z}_p^{\times}$, and  
\[
\widetilde{R}^{\rm ord}:=R^{\rm ord} \otimes_{\Lambda_1,\sigma} \Lambda_1,  
\]
a finite $\Lambda_1$-algebra with the structure homomorphism $\lambda \mapsto 1 \otimes \lambda$ on $\Lambda_1$. 
If $\widetilde{P}_0 \in \mathfrak{X}(\widetilde{R}^{\rm ord})$ is an arithmetic point lying 
over $P_0
$, 
there exist a formal Fourier expansion ${\bf F}$ with coefficients in the localization $
(\widetilde{R}^{\rm ord})_{(\widetilde{P}_0)}$ of $\widetilde{R}^{\rm ord}$ at $\widetilde{P}_0$, and a choice of $p$-adic periods $\Omega_{P} \in \overline{\mathbb{Q}}_p$ for $P \in \mathfrak{U}_0
$ in the sense of Greenberg-Stevens \cite{G-S93} satisfying the following properties: 
\begin{itemize}
\item
$\Omega_{P_0} \ne 0$. 
\item
There exists a normalization of ${\rm Lift}^{(2n)}(f_{2k})$ such that for each arithmetic point $\widetilde{P} \in \mathfrak{X}(\widetilde{R}^{\rm ord})$ lying over 
some arithmetic point $P \in \mathfrak{U}_0
$ 
of weight $2k$ with trivial Nebentypus, we have 
\[
{\bf F}(\widetilde{P})={\Omega_{P} \over \Omega^{\epsilon}(P)}\,{\rm Lift}^{(2n)}(f_{2k})^{*} \ne 0, 
\] 
where $\Omega^{\epsilon}(P) \in \mathbb{C}^{\times}$ 
is the non-zero complex period of $f_{2k}$ with signature $\epsilon \in \{\pm\}$ in the sense of Manin \cite{Man73} and Shimura \cite{Shim82}. 
\end{itemize}
\end{enumerate}
Therefore we obtain a 
$p$-adic analytic family of non-zero Siegel cusp forms $\{{\Omega_{P} \over \Omega^{\epsilon}(P)}\,
{\rm Lift}^{(2n)}(f_{2k})^{*}\}$ 
parametrized by varying weights $2k \ge 2k_{0}$ with 
\begin{center}
$k \equiv k_0 \pmod{2}$ and\,  $2k \equiv 2k_0 \pmod{(p-1)p^{m-1}}$ 
\end{center}
for some sufficiently large $m \ge 1$. 
\end{thm}

We note that the existence of such $p$-adic analytic families have been suggested by Guerzhoy \cite{Gue00} for $n=1$, and conjectured by Panchishkin \cite{Pan10} for arbitrary $n>1$. 
In particular, Guerzhoy \cite{Gue00} derived  
a similar $p$-adic interpolation property of an essential part of the Fourier expansion of ${\rm Lift}^{(2)}(f_{2k})$ under mild conditions. 
The above theorem can be regarded as a generalization of his result reformulated in the way that seems most appropriate for the study of $p$-adic properties 
of the Duke-Imamo{\={g}}lu lifting. 

\vspace*{3mm}
For the proof of Theorem 1.3, 
we will give an explicit form of the Fourier expansion of ${\rm Lift}^{(2n)}(f_{2k})^{*}$ (cf. Theorem 4.1 below).  
By combining this with the $\Lambda$-adic Shintani lifting due to Stevens \cite{Ste94}, 
we will give a $\Lambda$-adic analogue of the classical Duke-Imamo{\={g}}lu lifting 
for the universal ordinary $p$-stabilized newform ${\bf f}_{\rm ord}$, which 
allows us to resolve the $p$-adic interpolation problem for the whole Fourier expansion of ${\rm Lift}^{(2n)}(f_{2k})^{*}$ (cf. Theorem 4.4 below). 

\begin{rem}
From a geometric point of view, a $p$-adic deformation theory for Siegel modular forms of arbitrary genus has been established by Hida in the ordinary case (cf. \cite{Hid02, Hid04}, see also \cite{T-U99}).  
Unfortunately, ${\rm Lift}^{(2n)}(f_{2k})$ does not admit the ordinary $p$-stabilization in the sense of Hida. 
%
However, it turns out that a slight weaker version ${\rm Lift}^{(2n)}(f_{2k})^{*}$ 
is sufficient to adapt the ordinary theory. 
In the same spirit as Skinner-Urban \cite{S-U06}, we refer to it 
as the 
``semi-ordinary'' 
$p$-stabilization of ${\rm Lift}^{(2n)}(f_{2k})$. 
For further details on the topic will be discussed in \S 4 below. 
\end{rem}

When $n=1$, more generally in the same direction, Skinner-Urban \cite{S-U06} 
produced a $p$-adic deformation of the cuspidal automorphic representation of ${\rm GSp}_4(\mathbb{A}_{\mathbb{Q}})$ generated by 
${\rm Lift}^{(2)}(f)$ 
within the framework of a significant study of 
the Selmer group ${H_f^1(\mathbb{Q},\,V_f(k))}$ defined by Bloch-Kato \cite{B-K90}, where $V_f$ denotes the $p$-adic Galois representation associated with $f$ in the sense of Eichler-Shimura and Deligne (cf. \cite{Del69}). 
Although they must be not entirely smooth (e.g. we cannot associate Siegel modular forms of genus $2n \ge 4$ with Galois representations so far), similar arithmetic applications of 
 $p$-adic analytic families would be stimulated by  
the recent progress on 
Ikeda's generalization of the Duke-Imamo{\={g}}lu lifting 
towards a Langlands functorial lifting of cuspidal automorphic representations of ${\rm PGL}_2(\mathbb{A}_K)$ to ${\rm Sp}_{4n}(\mathbb{A}_K)$ over a totally real field extension $K/\mathbb{Q}$ (cf. \cite{Ike10}). 
Indeed, as a consequence of the period relation for ${\rm Lift}^{(2n)}(f)$ that was conjectured by Ikeda \cite{Ike06} and afterwards was shown by Katsurada and the author \cite{KK10}, we may produce a certain kind of congruence properties occurring between ${\rm Lift}^{(2n)}(f)$ and some genuine Siegel cusp forms of genus $2n$ under mild conditions, which is very similar to those in \cite{S-U06} (cf. \cite{Kat08,Kat10}, see also Brown \cite{Bro07} for $n=1$).
This type of congruence properties and their applications were firstly conjectured by Harder \cite{Har93} for the Saito-Kurokawa lifting 
and by Doi-Hida-Ishii \cite{DHI98} for the Doi-Naganuma lifting. 

\vspace*{2mm}
It should be emphasized that our approach based on the description of Fourier expansions is more explicit than the method using the theory of automorphic representations, and hence 
yields some practical benefit. For instance, the semi-ordinary $p$-stabilized form ${\rm Lift}^{(2n)}(f_{2k})^{*}$ can be regarded naturally as the Duke-Imamo{\={g}}lu lifting of 
$f_{2k}^{*}$, although the $p$-local component of the associated cuspidal automorphic representation of ${\rm Sp}_{4n}(\mathbb{A}_{\mathbb{Q}})$ is 
a quadratic twist of the Steinberg representation in general {(cf. \cite{Ike10})}. 

\vspace*{2mm}
Finally, as will be explained in \S 5, we note that the method we use for the Duke-Imamo{\={g}}lu lifting 
is adaptable to the Eisenstein series as well. Indeed, by taking Serre's $p$-adic analytic family of ordinary $p$-stabilized Eisenstein series instead of the Hida family, we also obtain a similar result for the Siegel Eisenstein series of genus $2n$, which is closely related to the results due to Takemori \cite{Tak10} and 
Panchishkin \cite{Pan00}. In addition, as conjectured in \cite{Pan10}, our constructions of $p$-adic analytic families are also extendable to those of families of Siegel cusp forms of odd genus 
by means of the Miyawaki lifting (cf. \cite{Ike06}). The corresponding result will appear elsewhere. 

\vspace*{3mm}
\begin{acknow}
The author 
is deeply grateful to Professor Alexei Panchishkin and Professor Haruzo Hida for their valuable suggestions and comments. Moreover, discussions with Professor Jacques Tilouine, Professor Hidenori Katsurada, Professor Siegfried B{\"o}cherer, Professor Tamotsu Ikeda, 
Professor Marc-Hubert Nicole 
and Dr. Tomokazu Kashio
have been illuminating for him. 
\end{acknow}

\smallskip

\begin{notation}
We denote by $\mathbb{Z},\, \mathbb{Q},\, \mathbb{R},$ and $\mathbb{C}$ the ring of integers, fields of rational numbers, real numbers and complex numbers, 
respectively. We put $e(x)=\exp(2\pi \sqrt{-1}x)$ for $x \in \mathbb{C}$. 
For each rational prime $l$, we denote by $\mathbb{Q}_l$, $\mathbb{Z}_l$ and $\mathbb{Z}_l^{\times}$ 
the field of $l$-adic numbers, the ring of $l$-adic integers and the group of units of $\mathbb{Z}_l$, respectively. Hereinafter, we fix an algebraic closure $\overline{\mathbb{Q}}_l$ of $\mathbb{Q}_l$. 
Let $v
_l(*)$ denote the $l$-adic valuation 
normalized as $v_l(l)=1$, and $e_l(*)$ the continuous additive character on $\overline{\mathbb{Q}}_l$ such that  
$e_l(y)=e(y)$ for all $y \in \mathbb{Q}$. 
Throughout the article, we fix an odd prime $p > 5$. 
From now on, we take the algebraic closure $\overline{\mathbb{Q}}$ of $\mathbb{Q}$ inside $\mathbb{C}$, and identify it with the image under a fixed embedding $\overline{\mathbb{Q}} \hookrightarrow \overline{\mathbb{Q}}_p$ once for all. 

For each integer 
$g \ge 1$, let ${\rm GSp}_{2g}
$ and ${\rm Sp}_{2g}
$ be 
the $\mathbb{Q}$-linear algebraic groups introduced as follows: 
\begin{enumerate}
\item[] 
$
{\rm GSp}_{2g}
:=
\left\{\, M \in {\rm GL}_{2g}\,
\right|  \left. {}^t M 
J
M = \mu(M) 
J
\textrm{ for some } \mu(M) \in 
{\rm GL}_1
\,\right\}
$,
\item[] 
${\rm Sp}_{2g}
:=\left\{\, M \in {\rm GSp}_{2g}
\left| \,\,\mu(M) =1\, \right. \right\}$, 
\end{enumerate}
where 
$
J=J_{2g} =
\left(
\begin{smallmatrix}
0_g
 & 1_g \\
{-1}_g & 
0_g
\end{smallmatrix}
\right) 
$ with the $g\times g$ unit (resp. zero) matrix $1_g$ {(resp. $0_g$)}. 
We denote by ${\rm B}_{2g}$ the standard Borel subgroup of ${\rm GSp}_{2g}$, and by 
${\rm P}_{2g}$ the associated Siegel parabolic subgroup, that is, 
${\rm P}_{2g}=\{M = \left(\begin{smallmatrix}
* & * \\
0_g & *
\end{smallmatrix}\right)\in {\rm GSp}_{2g}\}$. 
Each real point $M=\left(\begin{smallmatrix}
A & B \\
C & D
\end{smallmatrix}
\right) \in {\rm GSp}_{2g}(\mathbb{R})$ with $A,\,B,\,C,\,D \in {\rm Mat}_{g\times g}(\mathbb{R})$ and $\mu(M)>0$ acts on the Siegel upper-half space 
\[
\mathfrak{H}_g := \left\{ Z =X+\sqrt{-1}\,Y\in {\rm Mat}_{g\times g}(\mathbb{C}) \left|\, {}^t Z =Z,\, 
Y > 0\, (\textrm{positive definite})\right.\right\}
\]
of genus $g$ via the linear transformation $Z \mapsto M(Z)=(AZ+B)(CZ+D)^{-1}$. 
Then for a given $\kappa \in \mathbb{Z}$
and a function $F$ on $\mathfrak{H}_n$, we define an action of $M$ on $F$ by 
\[
(F|_
\kappa M )(Z):= \mu(M)^{g\kappa-g(g+1)/2} 
\det(CZ+D)^{-\kappa} F(M(Z)).
\] 
For handling Siegel modular forms of genus $g$, we consider the following congruence subgroups of the Siegel modular group 
${\rm Sp}_{2g}(\mathbb{Z})$: for each integer $N \ge 1$, put 
\begin{enumerate}[\upshape (i)]
\item[] $\Gamma_0(N)
:=\left\{ 
M \in {\rm Sp}_{2g}(\mathbb{Z}) \left|\,\,
M \equiv 
\left(\begin{array}{cc}
* & * \\
0_{g} & *
\end{array}\right) \pmod{N} 
\right. \right\}$,\\[1mm]

\item[] $\Gamma_1(N)
:=\left\{ 
M \in {\rm Sp}_{2g}(\mathbb{Z}) \left|\,\,
M \equiv 
\left(\begin{array}{cc}
* & * \\
0_{g} & 1_{g}
\end{array}\right) \pmod{N} 
\right. \right\}$. 
\end{enumerate} 
For each $\kappa \in \mathbb{Z}
$, 
the space $\mathscr{M}_{\kappa}(
\Gamma_1(N)
)^{(g)}
$ of ({\it holomorphic}\,) {\it Siegel modular forms} of weight $\kappa$ on $
\Gamma_1(N)\subseteq {\rm Sp}_{2g}(\mathbb{Z})
$, consists of $\mathbb{C}$-valued holomorphic functions $F$ on $\mathfrak{H}_{g}$ satisfying the following conditions: 
\begin{center}
\vspace*{-4mm}
\begin{minipage}[t]{0.975\textwidth}
\begin{enumerate}[\upshape (i)]

\item $F|_\kappa M = F$ \, for any 
$M
\in 
\Gamma_1(N)
$;

\item For each $M \in {\rm Sp}_{2g}(\mathbb{Z})$, the function $F|_\kappa M$ possesses a Fourier expansion of the form
\[
(F|_\kappa M)(Z) = 
\displaystyle\sum_
{T \in {\rm Sym}_g^{*}(\mathbb{Z})}
A_{T}(F|_\kappa M)\, e({\rm trace}(TZ)),
\]

\item[\hspace*{6mm}] 
where we denote by ${\rm Sym}_g^{*}(\mathbb{Z})$ the dual lattice of ${\rm Sym}_g(\mathbb{Z})$, that is, 
consisting of all half-integral symmetric matrices:   
\[
\hspace*{5mm}
{\rm Sym}_g^{*}(\mathbb{Z})=
\{T=(t_{ij}) \in {\rm Sym}_{g}(\mathbb{Q})\,|\,
t_{ii},\,2t_{ij} \in \mathbb{Z}\,\,(1 \leq i < j \leq g)
\}.\] 
It satisfies that $A_T(F |_\kappa M)=0$ unless $T\ge 0$ $(\textrm{semi positive definite})$ for all $M \in {\rm Sp}_{2g}(\mathbb{Z})$. 
\end{enumerate}
\end{minipage}
\end{center}
A modular form $F \in \mathscr{M}_{\kappa}(
\Gamma_1(N))^{(g)}$ is said to be {\it cuspidal} (or a {\it cusp form}) if it satisfies a stronger condition $A_{T}(F|_\kappa M) = 0$ unless $T > 0$ 
for all $M \in {\rm Sp}_{2g}(\mathbb{Z})$. We denote by $\mathscr{S}_{\kappa}(
\Gamma_1(N))^{(g)}
$ the subspace of $\mathscr{M}_{\kappa}(
\Gamma_1(N))^{(g)}$ consisting of all cusp forms. 
When $N=1$, we subsequently write $\mathscr{M}_{\kappa}({\rm Sp}_{2g}(\mathbb{Z}))$ and $\mathscr{S}_{\kappa}({\rm Sp}_{2g}(\mathbb{Z}))$ instead of $
\mathscr{M}_{\kappa}(\Gamma_1(1))^{(g)}$ and $
\mathscr{S}_{\kappa}(\Gamma_1(1))^{(g)}$, respectively. 
For each Dirichlet character $\chi$ modulo $N$, we denote by  
$\mathscr{M}_{\kappa}(\Gamma_0(N),\,\chi)^{(g)}$ the subspace of $\mathscr{M}_{\kappa}(\Gamma_1(N))^{(g)}$ consisting of all forms $F$ with Nebentypus $\chi$, that is, satisfying the condition
\[
F|_\kappa M = \chi(\det D) F \textrm{ \hspace{2mm}  for all } M
=\left(
\begin{smallmatrix}
  A & B  \\
  C & D  \\
\end{smallmatrix}
\right) 
\in \Gamma_0(N),  
\]
and put $\mathscr{S}_{\kappa}(\Gamma_0(N),\,\chi)^{(g)}:= \mathscr{M}_{\kappa}(\Gamma_0(N),\,\chi)^{(g)}\cap \mathscr{S}_{\kappa}(\Gamma_1(N))^{(g)}$. 
In particular, if $\chi=\chi_0$ is the principal character, we naturally write $\mathscr{M}_{\kappa}(\Gamma_0(N))^{(g)}=\mathscr{M}_{\kappa}(\Gamma_0(N), \chi_0)^{(g)}$ and ${\mathscr{S}_{\kappa}(\Gamma_0(N))^{(g)}=\mathscr{S}_{\kappa}(\Gamma_0(N),\,\chi_0)^{(g)}}$, respectively. 

For each $T=(t_{ij}) \in {\rm Sym}_g^{*}(\mathbb{Z})$ and $Z=(z_{ij}) \in \mathfrak{H}_g$, we write 
\[
q^{T}:=e({\rm trace}(TZ))
=\prod_{i=1}^{g} q_{ii}^{t_{ii}} \prod_{i<j \leq g} q_{ij}^{2t_{ij}}, 
\]
where $q_{ij}=e(z_{ij})\, (1\leq i \leq j \leq n)$. 
Then it follows from the definition that each $F \in \mathscr{M}_{\kappa}(
\Gamma_1(N)
)^{(g)}$ possesses the usual Fourier 
expansion 
\[
F(Z)=\sum_{\scriptstyle T \in {\rm Sym}_g^{*}(\mathbb{Z}), 
\atop {\scriptstyle T \ge 0}
} A_{T}(F)\, 
q^{T} \in 
\mathbb{C}[\,q_{ij}^{\pm 1}\,|\,1 \leq i < j \leq g\,][[q_{11},\cdots,q_{gg}]].    
\]
For each ring $R$, 
we write $R[[q
]]^{(g)}:=R[\,q_{ij}^{\pm 1}\,|\,1 \leq i < j \leq g\,][[q_{11},\cdots,q_{gg}]]$, 
which is a generalization of Serre's ring 
$R[[q
]]^{(1)}=R[[q]]$ consisting of all formal $q$-expansions with coefficients in $R$. 
In particular, if $F\in \mathscr{M}_{\kappa}(
\Gamma_1(N)
)^{(g)}$ is a Hecke eigenform (i.e. a simultaneous eigenfunction of all Hecke operators with similitude prime to $N$), then it is well-known that 
the field $K_F$ obtained by adjoining all Fourier coefficients of $F$ to $\mathbb{Q}$ is a totally real algebraic field of finite degree, to which we refer as the {\it Hecke field} of $F$.  
Hence we have \[
F \in K_F[[q
]]^{(g)} \subset \overline{\mathbb{Q}}
[[q
]]^{(g)} \hookrightarrow \overline{\mathbb{Q}}_p
[[q
]]^{(g)}. 
\]
For further details on Siegel modular forms set out above, see \cite{A-Z95} or \cite{Fre83}. 
\end{notation}

\section{Preliminaries}


\subsection{Classical Duke-Imamo{\={g}}lu lifting}

In this subsection, we review Ikeda's construction of the Duke-Imamo{\={g}}lu lifting for elliptic cusp forms (cf. \cite{Ike01}) and Kohnen's description of its Fourier expansion (cf. \cite{Koh02}). 
 
\vspace*{2.5mm}
To begin with, we recall some basic facts on elliptic modular forms of half-integral weight which were initiated by Shimura. 
For each $M 
\in \Gamma_0(4) \subset {\rm SL}_2(\mathbb{Z})$ and $z \in \mathfrak{H}_1$, put 
\[
j(M,z):={\theta_{1/2}(M(z))\over \theta_{1/2}(z)}, 
\]
where $\theta_{1/2}(z)=\sum_{m \in \mathbb{Z}} \mathbf{e}({m^2}z)$ is the standard theta function. It is well-known that $j(M,z)$ for $M \in \Gamma_0(4)$ satisfies a usual 1-cocycle relation, and hence defines a factor of automorphy. 
%
Then for each integers $k,\,N\ge1$, a 
$\mathbb{C}$-valued holomorphic function $h$ on $\mathfrak{H}_1$ 
is called a 
modular form of weight $k+1/2$ on $\Gamma_0(4N)$ if it 
satisfies the following conditions similar to those in the integral weight case: 
\begin{center}
\vspace*{-3mm}
\begin{minipage}[t]{0.975\textwidth}
\begin{enumerate}[\upshape (i)]

\item $(h|_{k+1/2}M) (z):=j(M,z)^{-2k-1} h(M(z))=h(z)$
for any $M \in \Gamma_0(4N)$; 

\item \vspace*{0.5mm}For each $M \in {\rm SL}_2(\mathbb{Z})$, the form $h|_{k+1/2}M$ has a Fourier expansion
\[
(h|_{k+1/2}M)(z)=\!\displaystyle\sum_{\scriptstyle m=0 
}^{\infty} c_m (h|_{k+1/2}M)\, q^m,  
\]
where $q=
e(z)$.  
\end{enumerate}
\end{minipage}
\end{center}
In particular, a modular form $h$ is said to be cuspidal (or a cusp form) 
if it satisfies the condition $c_0(h|_{k+1/2}M)=0$ for all $M \in {\rm SL}_2(\mathbb{Z})$. 
We denote by $\mathscr{M}_{k+1/2}(\Gamma_0(4N))^{(1)}$ and 
$\mathscr{S}_{k+1/2}(\Gamma_0(4N))^{(1)}$ 
the space consisting of all modular forms of weight $k+1/2$ on $\Gamma_0(4N)$ and its cuspidal subspace, respectively. 

\vspace*{2.5mm}
As one of the most significant properties of such forms of half-integral weight, Shimura \cite{Shim73} established that there exists a Hecke equivariant linear correspondence between $\mathscr{M}_{k+1/2}(\Gamma_0(4N))^{(1)}$ and $\mathscr{M}_{2k}(\Gamma_0(N))^{(1)}$, to which we refer as the {\it Shimura correspondence}. 
More precisely, Kohnen introduced 
the plus spaces $\mathscr{M}_{k+1/2}^{+}(\Gamma_0(4N))^{(1)}$ 
and $\mathscr{S}_{k+1/2}^{+}(\Gamma_0(4N))^{(1)}$ respectively to be the subspaces of $\mathscr{M}_{k+1/2}(\Gamma_0(4N))^{(1)}$ and $\mathscr{S}_{k+1/2}(\Gamma_0(4N))^{(1)}$ consisting of all forms $h$ with  
\[
c_m(h)=0 \textrm{ unless }(-1)^k m \equiv 0 \textrm{ or }1 \pmod{4}, 
\]
and showed that if either $k \ge 2$ or $k=1$ and $N$ is cubefree, 
the Shimura correspondence gives the diagram of linear isomorphisms  
\[
\begin{array}{ccc}
\mathscr{M}_{k+1/2}^{+}(\Gamma_0(4N))^{(1)} &\stackrel{\simeq}{\rightarrow}& \mathscr{M}_{2k}(\Gamma_0(N))^{(1)} \\
\cup & {} & \cup \\
\mathscr{S}_{k+1/2}^{+}(\Gamma_0(4N))^{(1)} & \stackrel{\simeq}{\rightarrow} &\mathscr{S}_{2k}(\Gamma_0(N))^{(1)}, 
\end{array}
\]
which is commutative with the actions of Hecke operators (cf. \cite{Koh80, Koh81, Koh82}). 
When $N=1$, 
the Shimura correspondence can be characterized explicitly in terms of Fourier expansions as follows: 
If $f=\sum_{m=1}^{\infty} a_m(f) q^m \in \mathscr{S}_{2k}(\mathrm{SL}_2(\mathbb{Z}))$ is a Hecke eigenform normalized as $a_1(f)=1$, and 
\[
h=\sum_{\scriptstyle m \ge 1, \atop {\scriptstyle (-1)^k m \equiv 0,1 \!\!\!\!\pmod{4}}} c_m(h) q^m \in \mathscr{S}_{k+1/2}^{+}(\Gamma_0(4))^{(1)}
\] 
corresponds to $f$ via the Shimura correspondence, then for each fundamental discriminant $\mathfrak{d}$ (i.e. 
$\mathfrak{d}$ is either $1$ or the discriminant of a quadratic field) with $(-1)^k \mathfrak{d} >0$, and $1 \le m \in \mathbb{Z}$, we have 
\begin{equation}
c_{
|\mathfrak{d}| m^2}(h)=c_{
|\mathfrak{d}|}(h) \sum_{d | m} \mu(d) \left({\,\mathfrak{d}\, \over d}\right) d^{k-1} a_{(m/d)}(f), 
\end{equation}
where $\mu(d)$ is the M{\"o}bius function, and $\left({\,\mathfrak{d}\, \over d}\right)$ the Kronecker character corresponding to $\mathfrak{d}$. We note that the inverse correspondence of the Shimura correspondence is 
determined uniquely up to scalar multiplication. It is because unlike the integral weight case, there is no canonical normalization of half-integral weight forms. 
Hence we should choose a suitable normalization of it in accordance with the intended use. 

\begin{rem}
As will be explained more precisely in \S\S 3.2 below, for integers $N \ge1$, $k \ge 2$ and a fundamental discriminant $\mathfrak{d}$ with $(-1)^k \mathfrak{d} >0$, Shintani \cite{Shin75} and Kohnen \cite{Koh81} constructed a theta lifting 
\[
\vartheta_{\mathfrak{d}} : \mathscr{S}_{2k}(\Gamma_0(N))^{(1)} \longrightarrow \mathscr{S}_{k+1/2}^{+}(\Gamma_0(4N))^{(1)},
\]
which 
gives an inverse correspondence of the Shimura correspondence admitting an algebraic normalization 
with respect to $(-1)^k \mathfrak{d}$. 
\end{rem}

\vspace*{2mm}On the other hand, for each integers $n$, $k \ge 1$ with $k> n+1$ and $n \equiv k \pmod{2}$, 
we define the ({\it holomorphic}) {\it Siegel Eisenstein series} of weight $k+n$ on ${\rm Sp}_{4n}(\mathbb{Z})$ as follows: 
for each $Z \in \mathfrak{H}_{2n}$, put 
\begin{eqnarray*}
E_{k+n}^{(2n)}(Z)&:=& 2^{-n} \zeta(1-k-n) \prod_{i=1}^{n} \zeta(1-2k-2n+2i) \\
&{}& \times \sum_{M=\left(\begin{smallmatrix}
* & * \\
C & D
\end{smallmatrix}\right) \in {\rm P}_{4n} \cap {\rm Sp}_{4n}(\mathbb{Z}) \backslash {\rm Sp}_4(\mathbb{Z})} \det(CZ+D)^{-k-n}.
\end{eqnarray*}
It is well-known that $E_{k+n}^{(2n)}$ is a non-cuspidal Hecke eigenform in $\mathscr{M}_{k+n}({\rm Sp}_{4n}(\mathbb{Z}))$. 
In addition, for each $T \in {\rm Sym}_{2n}^{*}(\mathbb{Z})$ with $T>0$, 
we decompose the associated discriminant $\mathfrak{D}_T:=(-1)^n \det(2T)$ 
into the form 
$$\mathfrak{D}_T = \mathfrak{d}_T\,\mathfrak{f}_T^2$$ with the fundamental discriminant 
$\mathfrak{d}_T$ corresponding to the quadratic field extension $\mathbb{Q}(\sqrt{\mathfrak{D}_T})/\mathbb{Q}
$ and an integer $\mathfrak{f}_T \ge 1$. 
Then the 
Fourier coefficient $A_T(E_{k+n}^{(2n)})$ 
is taken of the following form: 
\begin{equation}
A_T(E_{k+n}^{(2n)})=
L(1-k,\left({\mathfrak{d}_T \over *}\right)) 
\displaystyle\prod_{l | \mathfrak{f}_T } 
F_l(T;\, l^{k-n-1}
 ), 
\end{equation}
where $L(s,\left({\mathfrak{d}_T \over *}\right))
:=\sum_{m =1}^{\infty}\left({\mathfrak{d}_T \over m}\right)m^{-s}
$, 
and for each prime $l$, 
$F_l(T;\, X)
$ denotes the polynomial in one variable $X$ with coefficients in $\mathbb{Z}$ appearing in the factorization of the formal power series
\[
b_l(T;\,X):=
\sum_{R \in {\rm Sym}_{2n}(\mathbb{Q}_l)/{\rm Sym}_{2n}(\mathbb{Z}_l)} e_l({\rm trace}(TR)) 
X^{v_l(\nu_R)},
\]
where $\nu_R=[\mathbb{Z}_l^{2n} + \mathbb{Z}_l^{2n} R : \mathbb{Z}_l^{2n}]$, that is, 
\begin{equation}
b_l(T;\,X)
=
\displaystyle{(1-X)
\prod_{i=1}^{n} (1-l^{2i} X^2) \over 1- \left({\mathfrak{d}_T \over l}\right)l^n X}\, F_l(T;\,X)
\end{equation}
(cf. \cite{Shim73, Shim94b, 
Kit84,
Fei86}). 
%
Moreover, it is known that 
$F_l(T;\,X)$ is the polynomial of degree $2v_l(\mathfrak{f}_T)$ with $F_l(T;\,0)=1$ 
and satisfies the functional equation 
\begin{equation}
F_l(T;\, l^{-2n-1} X^{-1})=(l^{2n+1} X^2)^{- v_l(\mathfrak{f}_T)} F_l(T;\, X)
\end{equation}
(cf. \cite{Kat99}). 
In particular, we have $F_l(T;\,X)=1$ if $v_l(\mathfrak{f}_T)=0$. 
We easily see that $F_l(uT;\, X)=F_l(T;\,X)$ for each $u \in \mathbb{Z}_l^{\times}$. 

\begin{rem}
For each prime $l$, the formal power series $b_l(T;\,X)$ gives rise to the local Siegel series $b_l(T;\,s):=b_l(T;\,l^{-s})$ for $s \in \mathbb{C}$ with ${\rm Re}(s)>0$. In particular, for each even integer $\kappa 
 > 2n+1$, then the value 
$b_l(T;\,\kappa)$ coincides with the local density 
\begin{eqnarray*}
\lefteqn{
\alpha_l(T,\,H_{2\kappa}):=\lim_{r \to \infty} (l^r)^{-4n\kappa+n(2n+1)}
} \\
&& \times \# \!\left\{\,
U \in {\rm Mat}_{2\kappa \times 2n}(\mathbb{Z}_l/l^r \mathbb{Z}_l)
\, |\, 
{}^t U H_{2\kappa} U - T
\in l^r {\rm Sym}_{2n}^{*}(\mathbb{Z}_l) \,
\right\}, 
\end{eqnarray*}
where $H_{2\kappa}={1 \over \,2\,} \left(\begin{smallmatrix}
0_{\kappa} & 1_{\kappa} \\
1_{\kappa} & 0_{\kappa}
\end{smallmatrix}
\right)$. In this connection, there have been numerous papers focusing on the local densities of quadratic forms. 
\end{rem}

\vspace*{2mm}
Following \cite{Ike01}, we construct the the Duke-Imamo{\={g}}lu lifting as follows: 

\begin{ikedathm}[Theorems 3.2 and 3.3 in \cite{Ike01}]
Suppose that $n$, $k$ are positive integers with $n \equiv k \pmod{2}$.  
Let $f=\sum_{m=1}^{\infty} a_m(f) q^m \in \mathscr{S}_{2k}(\mathrm{SL}_2(\mathbb{Z}))$ be a normalized Hecke eigenform, and $h=\sum_{m \ge 1} c_m(h) q^m \in \mathscr{S}_{k+1/2}^{+}(\Gamma_0(4))^{(1)}$ a corresponding Hecke eigenform as in $(2)$. 
Then for each $0 < T \in {\rm Sym}_{2n}^{*}(\mathbb{Z})$ with discriminant $\mathfrak{D}_T
=\mathfrak{d}_T \,\mathfrak{f}_T^2$, put  
\begin{equation}
A_{T}({\rm Lift}^{(2n)}(f)) := c_{
|\mathfrak{d}_T|
}(h) \prod_{l | \mathfrak{f}_T} 
\alpha_l(f)^{v_l(\mathfrak{f}_T)} F_l(T;\, l^{-k-n}\beta_l(f)), 
\end{equation}
where for each prime $l$, we denote by $(\alpha_l(f),\beta_l(f))$ the ordered pair of the roots of 
$
X^2 - a_l(f) X + l^{2k-1} =0 
$
with $v_l(\alpha_l(f)) \le v_l(\beta_l(f))$. 
Then the Fourier expansion 
\[
{\rm Lift}^{(2n)}(f):=\displaystyle\sum_{\scriptstyle T \in {\rm Sym}_{2n}^{*}(\mathbb{Z}), 
\atop {\scriptstyle T > 0}
} A_{T}({\rm Lift}^{(2n)}(f))\, q^{T} 
\]
gives rise to a Hecke eigenform in $\mathscr{S}_{k+n}({\rm Sp}_{4n}(\mathbb{Z}))$ such that 
\[
L(s,\, {\rm Lift}^{(2n)}(f),\,{\rm st})=\zeta(s)\prod_{i=i}^{2n}L(s+k+n-i,\, f).
\]
\end{ikedathm}


\begin{rem}
For a given Hecke eigenform $F \in \mathscr{M}_{k+n}({\rm Sp}_{4n}(\mathbb{Z}))$ with the Satake parameter $(\psi_0(l),\psi_1(l), \cdots, \psi_{2n}(l)) \in (\mathbb{C}^{\times})^{2n+1}/W_{2n}$ 
for each prime $l$, then the spinor $L$-function $L(s,\, F,\,{\rm spin})$ and the standard $L$-function $
L(s,\, F,\,{\rm st})$ associated with $F$ are respectively defined as follows:
\begin{eqnarray*}
\lefteqn{
\hspace*{1.5mm}
L(s,\, F,\,{\rm spin})} \\
&:=&
\displaystyle\prod_{l < \infty} 
\left\{(1- \psi_0(l)l^{-s})\prod_{r=1}^{2n} \prod_{1 \le i_1 < \cdots < i_r \le 2n}
(1-\psi_0(l)\psi_{i_1\!}(l) \cdots \psi_{i_r\!}(l) l^{-s})\right\}^{-1}, 
\end{eqnarray*}
\begin{flushleft}
\hspace*{1.5mm}
$L(s,\,F,\,{\rm st})
:= \displaystyle\prod_{l < \infty}
\left\{ (1-l^{-s})\prod_{i=1}^{2n} (1-\psi_i(l) l^{-s})(1-\psi_i(l)^{-1} l^{-s})\right\}^{-1}$,
\end{flushleft}
Then 
it follows from the explicit form of 
$L(s,\,{\rm Lift}^{(2n)}(f),\,{\rm st})$ 
and 
the fundamental equation 
$\psi_0(l)^2 \psi_1(l) \cdots \psi_{2n}(l) = l^{2n(k+n)-n(2n+1)}$ 
that  
the Satake parameter of ${\rm Lift}^{(2n)}(f)$ is taken as 
\begin{equation}
{
\psi_i(l)=\left\{
\begin{array}{ll}
l^{nk-n(n+1)/2} & \textrm{if }i=0, \\[1.5mm]
\alpha_l(f) l^{-k+i} & \textrm{if }1 \le i \le n, \\[2mm]
\beta_l(f) l^{-k+i-n} & \textrm{if }n+1 \le i \le 2n. 
\end{array}\right.
}
\end{equation}
Hence the spinor $L$-function $L(s,\,{\rm Lift}^{(2n)}(f),\,{\rm spin})$ can be also written explicitly in terms of 
the symmetric power $L$-functions $L(s,\,f,\,{\rm sym}^{r})$ of $f$ with some $0 \le r \le n$ (cf. \cite{Mur02, Sch10}). 
\end{rem}

\vspace*{1mm}
According to the equation (3), we may formally look at the Siegel Eisenstein series $E_{k+n}^{(2n)}$ as the Duke-Imamo{\={g}}lu lifting of the normalized elliptic Eisenstein series \[
E_{2k}^{(1)}=
{\zeta(1-2k) \over 2}+ \sum_{m =1}^{\infty} \sigma_{2k-1}(m) q^m \in \mathscr{M}_{2k}({\rm SL}_2(\mathbb{Z})), 
\] 
where $\sigma_{2k-1}(m)=\displaystyle\sum_{\scriptstyle 0<d | m
} d^{2k-1}$. Indeed, 
\vspace*{-3mm}we easily see that for each prime $l$, 
$$(\alpha_l(E_{2k}^{(1)}), \beta_l(E_{2k}^{(1)}))=(1,l^{2k-1}),$$
and it is well-known that 
Cohen's Eisenstein series 
$H_{k+1/2} 
\in \mathscr{M}_{k+1/2}^{+}(\Gamma_0(4))^{(1)}$ corresponds to $E_{2k}^{(1)}$ via the Shimura correspondence, which possesses the Fourier coefficient 
$$c_{|\mathfrak{d}|}(H_{k+1/2})=L(1-k,\,\left({\,\mathfrak{d}\, \over *}\right))$$  
for each fundamental discriminant $\mathfrak{d}$ with $(-1)^k \mathfrak{d}>0$ (cf. \cite{Coh75, E-Z85}). 

\vspace*{3mm}
We also note that the Duke-Imamo{\={g}}lu lifting 
does not vanishes identically. Indeed, for each $0 < T \in {\rm Sym}_{2n}^{*}(\mathbb{Z})$ with $\mathfrak{D}_T
=\mathfrak{d}_T$ (i.e. $\mathfrak{f}_T=1$), the equation (6) yields the equation 
\[
A_T({\rm Lift}^{(2n)}(f))=c_{|\mathfrak{d}_T|}(h). 
\] 
Hence the non-vanishing of $A_T({\rm Lift}^{(2n)}(f))$ is guaranteed by the well-studied non-vanishing theorem for 
Fourier coefficients of $h$ as follows: 
\begin{lem}
For each integer $k \ge 6$, let $f \in \mathscr{S}_{2k}({\rm SL}_2(\mathbb{Z}))$ 
and $h \in \mathscr{S}_{k+1/2}^{+}(\Gamma_0(4))^{(1)}$ be taken as above. 
Then there exists a fundamental discriminant $\mathfrak{d}$ 
with $(-1)^k \mathfrak{d}>0$
such that  
\[c_{
|\mathfrak{d}|}(h) \ne 0. \] 
Moreover, for each prime $p$, a similar statement remains valid under the additional condition either $\mathfrak{d} \equiv 0 \pmod{p}$ or $\mathfrak{d} \not\equiv 0 \pmod{p}$. 
\end{lem}

\begin{proof}
For each fundamental discriminant $\mathfrak{d}$ with $(-1)^k \mathfrak{d} >0$, 
Kohnen-Zagier \cite{K-Z81} established the equation  
\begin{equation}
{
c_{|\mathfrak{d}|}(h)^2 \over 
\|h\|^2 }
={(k-1)! \over \pi^k} |\mathfrak{d}|^{k-1/2}{L_{\mathfrak{d}}(k,\,f) \over 
\|f\|^2}, 
\end{equation}
where $L_{\mathfrak{d}}(s,f):=\sum_{m=1}^{\infty} \left({\,\mathfrak{d}\, \over m}\right) a_m(f) m^{-s}$, and we denote by $\|f\|^2$ and $\|h\|^2 $ the Petersson norms square of $f$ and $h$ respectively, that is,  
\begin{eqnarray*}
\|f\|^2&=\langle f,\,f \rangle:=&\displaystyle\int_{{\rm SL}_2(\mathbb{Z}) \backslash \mathfrak{H}_1} |f(x+\sqrt{-1}y)|^2 y^{2k-2}dxdy, \\ 
\|h\|^2&=\langle h,\,h \rangle:=&\displaystyle{1 \over \,6\,}\int_{\Gamma_0(4) \backslash \mathfrak{H}_1} |h(x+\sqrt{-1}y)|^2 y^{k-3/2}dxdy.
\end{eqnarray*}
Hence the existence of a fundamental discriminant $\mathfrak{d}$ 
with desired properties 
follows immediately from the non-vanishing theorem for $L_{\mathfrak{d}_T}(k,\,f)$ (cf. 
\cite{BFH90,Wal91}). We complete the proof. 
\end{proof}
 
\vspace*{3mm}
For the convenience in the sequel, we describe the Fourier expansion of ${\rm Lift}^{(2n)}(f)$ a little more precisely. 
For each prime $l$ dividing $\mathfrak{f}_T$, by virtue of the functional equation (5), we have 
\[
\alpha_l(f)^{v_l(\mathfrak{f}_T)} F_l(T;\, l^{-k-n}\beta_l(f))
=\beta_l(f)^{v_l(\mathfrak{f}_T)} F_l(T;\, l^{-k-n}\alpha_l(f)), 
\]
and hence $\alpha_l(f)^{v_l(\mathfrak{f}_T)} F_l(T, l^{-k-n}\beta_l(f))$ can be written in terms of $\alpha_l(f)+\beta_l(f)=a_l(f)$ and $l^{-k}\alpha_l(f)\beta_l(f)=l^{k-1}$. In fact, Kohnen showed 
that for each 
$l$, 
\begin{equation}
\alpha_l(f)^{v_l(\mathfrak{f}_T)} F_l(T;\, l^{-k-n}\beta_l(f))
=\sum_{i=0}^{v_l(\mathfrak{f}_T)} \phi_T(l^{v_l(\mathfrak{f}_T)-i}) 
(l^{k-1})^{v_l(\mathfrak{f}_T)-i} a_{l^{i}}(f), 
\end{equation}
with some arithmetic function $\phi_T(d)$ with values in $\mathbb{Z}$ defined for each integer $d\ge1$ dividing $\mathfrak{f}_T$ (cf. \cite{Koh02}). Hence we obtain another explicit form 
\[
A_{T}({\rm Lift}^{(2n)}(f)) = c_{|\mathfrak{d}_T|}(h) \prod_{l | \mathfrak{f}_T} 
\sum_{i=0}^{v_l(\mathfrak{f}_T)} \phi_T(l^{v_l(\mathfrak{f}_T)-i}) 
l^{(v_l(\mathfrak{f}_T)-i)(k-1)} a_{l^{i}}(f). 
\]
We will make use of this equation as well as (6) in the sequel. 

\begin{rem}
For a given $f \in \mathscr{S}_{2k}({\rm SL}_{2}(\mathbb{Z}))$, 
Ikeda's construction of 
${\rm Lift}^{(2n)}(f)$ 
obviously depends on the choice of 
$h \in \mathscr{S}_{k+1/2}^{+}(\Gamma_0(4))^{(1)}$. 
By combining the equation (2) with Kohnen's refinement 
of the Fourier expansion of ${\rm Lift}^{(2n)}(f)$, 
we may realize the Duke-Imamo{\={g}}lu lifting 
as an explicit linear mapping $\mathscr{S}_{k+1/2}^{+}(\Gamma_0(4))^{(1)} \to \mathscr{S}_{k+n}({\rm Sp}_{4n}(\mathbb{Z}))$. 
Moreover, Kohnen-Kojima \cite{K-Ko05} and Yamana \cite{Yam10} characterized the image of the mapping in terms of a relation between Fourier coefficients, 
which can be regarded as a generalization of Maass' characterization of 
${\rm Lift}^{(2)}(f)$. 
\end{rem}

\subsection{$\Lambda$-adic Siegel modular forms}
In this subsection, we introduce the notion of $\Lambda$-adic Siegel modular forms of arbitrary genus $g \ge 1$ 
from point of view of Fourier expansions. 

\vspace*{2mm}
Let $\varGamma=1+p\mathbb{Z}_p$ be the maximal torsion-free subgroup of $\mathbb{Z}_p^{\times}$. We choose and fix a topological generator $\gamma \in \varGamma$ such that $\varGamma=\gamma^{\mathbb{Z}_p}$. 
Let $\Lambda=\mathbb{Z}_p[[\varGamma]]$ and $\Lambda_1=\mathbb{Z}_p[[\mathbb{Z}_p^{\times}]]$ be the completed group rings on $\varGamma$ and on $\mathbb{Z}_p^{\times}$ over $\mathbb{Z}_p$, respectively. We easily see that $\Lambda_1$ has a natural $\Lambda$-algebra structure induced from the natural isomorphism $\Lambda_1 \simeq \Lambda[
\mu_{p-1}
]$, where $\mu_{p-1}$ denotes the maximal torsion subgroup consisting of all $(p-1)$-th roots of unity. 

\begin{rem}
As is well-known, $\Lambda$ is isomorphic to the power series ring $\mathbb{Z}_p[[X]]$ in one variable $X$ with coefficients in $\mathbb{Z}_p$ under 
$\gamma
\mapsto 1+X$. 
In addition, it is also known that $\mathbb{Z}_p[[X]]
$ is isomorphic to the ring ${\rm Dist}(\mathbb{Z}_p,\mathbb{Z}_p)$ consisting of all distributions on $\mathbb{Z}_p$ with values in $\mathbb{Z}_p$. 
Indeed, any distribution $\mu \in {\rm Dist}(\mathbb{Z}_p,\mathbb{Z}_p)$ corresponds to  
the power series
\[
A_{\mu}(X) =\int_{\mathbb{Z}_p} (1+X)^x d\mu(x) = \sum_{m=0}^{\infty} \int_{\mathbb{Z}_p} 
\binom{x}{m} 
d\mu(x) X^m \in \mathbb{Z}_p[[X]],
\]
where $\displaystyle\binom{x}{m}$ is the binomial function. 
Therefore we obtain \vspace*{-2mm}
\[
\Lambda \simeq \mathbb{Z}_p[[X]] \simeq {\rm Dist}(\mathbb{Z}_p,\mathbb{Z}_p), 
\]
which allows us to consider the definition of $\Lambda$-adic Siegel modular forms below from a different point of view.  
\end{rem}

To begin with, 
we introduce the $\Lambda$-adic analytic spaces as an alternative notion for the weights of holomorphic Siegel modular forms in the following: 

\begin{defn}[$\Lambda$-adic analytic spaces]
For each $\Lambda_1$-algebra $R$ finite flat over $\Lambda$, we define the 
{\it $\Lambda$-adic analytic space} $\mathfrak{X}(R)$ associated with $R$ as 
\[
\mathfrak{X}(R):={\rm Hom}_{\rm cont}(R, \overline{\mathbb{Q}}_p), 
\]
on which the following arithmetic data are introduced: 
\begin{itemize}
\item[(i)]
A point $P \in \mathfrak{X}(R)$ is said to be {\it arithmetic} if 
there exists an integer $\kappa \geq 2$ such that the restriction of $P$ to $\mathfrak{X}(\Lambda) 
:={\rm Hom}_{\rm cont}(\Lambda, \overline{\mathbb{Q}}_p)
\simeq {\rm Hom}_{\rm cont}(\varGamma, \overline{\mathbb{Q}}_p^{\times})
$ corresponds to a continuous character $
P_{\kappa} : \varGamma \to \overline{\mathbb{Q}}_p^{\times}$ satisfying
$
P_{\kappa}(\gamma) = \gamma^{\kappa}.
$
We denote by $\mathfrak{X}_{\rm alg}(R)$ the subset consisting of all arithmetic points in $\mathfrak{X}(R)$. 
\item[(ii)]
An arithmetic 
point $P \in \mathfrak{X}_{\rm alg}(R)$ is said to be of {\it signature} $(\kappa,\,\varepsilon)$ if 
there exist an integer $\kappa \geq 2$ and a finite character $\varepsilon : \mathbb{Z}_p^{\times} \to \overline{\mathbb{Q}}_p^{\times}$ such that $P$ lies over 
the point $P_{\kappa,\varepsilon} \in \mathfrak{X}_{\rm alg}(\Lambda_1)\simeq {\rm Hom}_{\rm cont}(\mathbb{Z}_p^{\times}, \overline{\mathbb{Q}}_p^{\times})$ corresponding to the character 
$P_{\kappa,\varepsilon}(y)= y^{\kappa}\varepsilon(y)$ on $\mathbb{Z}_p^{\times}$. 
For simplicity, we denote such $P$ by $P=(\kappa, \varepsilon)$ and often refer to it as the arithmetic point of {\it weight} $\kappa$ with {\it Nebentypus} $\varepsilon \omega^{-\kappa}$ 
in the sequel. 
\end{itemize}
\end{defn}

We note that $\mathfrak{X}(\Lambda)$ has a natural analytic structure induced from the identification ${{\rm Hom}_{\rm cont}(\Lambda, \overline{\mathbb{Q}}_p)
\simeq {\rm Hom}_{\rm cont}(\Gamma, \overline{\mathbb{Q}}_p)}$. Moreover, 
restrictions 
to $\Lambda_1$ and then to $\Lambda$
induce a surjective finite-to-one mapping 
\[
\pi : \mathfrak{X}(R) \twoheadrightarrow \mathfrak{X}(\Lambda_1) \twoheadrightarrow \mathfrak{X}(\Lambda), 
\]
which allows us to define analytic charts around all points of $\mathfrak{X}_{\rm alg}(R)$. 
Indeed, it is established by Hida 
that each $P \in \mathfrak{X}_{\rm alg}(R)$ is unramified over $\mathfrak{X}(\Lambda)$, and consequently 
there exists a natural local section of $\pi$ 
\[
S_{P}: U_{P} \subseteq \mathfrak{X}(\Lambda) \to \mathfrak{X}(R)
\]
defined on a neighborhood $U_{P}$ of $
\pi(P)
$ such that $S_{P}(
\pi(P))={P}$. These local sections 
endow $\mathfrak{X}(R)$ with analytic charts around points in $\mathfrak{X}_{\rm alg}(R)$. For each $P \in \mathfrak{X}_{\rm alg}(R)$, a function ${{\bf f} : \mathfrak{U} \subseteq \mathfrak{X}(R) \to \overline{\mathbb{Q}}_p}$ defined on $\mathfrak{U}=S_{P}(U_{P})$ is called {\it analytic} if ${{\bf f} \circ S_{P} : U_{P} 
\to \overline{\mathbb{Q}}_p}$ is analytic. 
In parallel, an open subset $\mathfrak{U} \subseteq \mathfrak{X}(R)$ containing some $P \in \mathfrak{X}_{\rm alg}(R)$ is called an {\it analytic neighborhood} of $P$ if $\mathfrak{U}=S_{P}(U_{P})$. 
For instance, we easily see that each element ${\bf a} \in R$ gives rise to an analytic function 
${\bf a} : \mathfrak{X}_{\rm alg}(R)\to \overline{\mathbb{Q}}_p$ defined by $\mathbf{a}(P)=P(\mathbf{a})$. 
In most generality, if $P \in \mathfrak{X}(R)$ is unramified over $\mathfrak{X}(\Lambda)$, then each 
element ${\bf a} \in R_{(P)}$ gives rise to an analytic function defined on some analytic neighborhood of $P$, where $R_{(P)}$ denotes the localization of $R$ at $P$, and gives rise to a discrete valuation ring finite and 
unramified over $\Lambda$ (cf. Corollary 1.4 in \cite{Hid86a}). 
From this point of veiw, we refer to the evaluation $\mathbf{a}(P)$ at $P \in \mathfrak{X}_{\rm alg}(R)$ 
as the {\it specialization} of $\mathbf{a}$ at $P$ according to the custom. 

\vspace*{3mm}
Following \cite{G-S93, Hid93} and \cite{Pan00}, we define the $\Lambda$-adic Siegel modular forms 
as follows: 

\begin{defn}[$\Lambda$-adic Siegel modular forms]
Let $R$ be a $\Lambda_1$-algebra finite flat over $\Lambda$.  
%
For each integer $g \ge 1$, 
pick $P_0 =({\kappa_0},\,\omega^{{\kappa_0}})\in \mathfrak{X}_{\rm alg}(R)$ with ${\kappa_0}> g+1$. 
A 
formal Fourier expansion 
\[
\mathbf{F} = \displaystyle\sum_{\scriptstyle T \in {\rm Sym}_g^{*}(\mathbb{Z}), 
\atop {\scriptstyle T \ge 0}
}
\mathbf{a}_{T}\,q^{T} \in R_{(P_0)}
[[q]]^{(g)}
\]
is called a {\it $\Lambda$-adic Siegel modular form} of genus $g$ and of level $1$ 
if there exists an analytic neighborhood $\mathfrak{U}_0$ of $P_0$ such that 
for each arithmetic point 
$P =(\kappa, \omega^{\kappa})\in \mathfrak{U}_0
$ 
with $\kappa \geq {\kappa_0}
$, 
the specialization 
\[
\mathbf{F}({P}) := \sum_{T} \mathbf{a}_{T}(P)
 \, q^{T}
 \in \overline{\mathbb{Q}}_p[[q]]^{(g)} 
\]
gives rise to the Fourier expansion of 
a holomorphic Siegel modular form in
$\mathscr{M}_{\kappa}(\Gamma_0(p))^{(g)}$. 
In particular, a $\Lambda$-adic Siegel modular form $\mathbf{F}$ 
is said to be {\it cuspidal} (or a {\it cusp form}\,) if $\mathbf{F}(P) \in \mathscr{S}_{\kappa}(\Gamma_0(p))^{(g)}$ 
for almost all $P \in \mathfrak{U}_0$. 
\end{defn}

If there exists a $\Lambda$-adic Siegel modular form ${\bf F} \in R_{(P_0)}[[q]]^{(g)}$, then each coefficient ${\bf a}_{T} \in R_{(P_0)}$ of ${\bf F}$ gives rise to an analytic function defined on $\mathfrak{U}_0
$. Hence every specialization ${\bf F}(P)$ gives a holomorphic Siegel modular form whose Fourier coefficients are $p$-adic analytic functions on $\mathfrak{U}_0$. In this context, we mean by a {\it $p$-adic analytic family}  
the infinite collection 
of holomorphic Siegel modular forms $\{{\bf F}(P) \in \mathscr{M}_{\kappa}(\Gamma_0(p))^{(g)}\}$ 
parametrized by varying arithmetic points $P =(\kappa, \omega^{\kappa})\in \mathfrak{U}_0$. 
In addition, by identifying such $P \in \mathfrak{U}_0$ 
with the element $(\kappa, \kappa\!\!\pmod{p-1})$ 
in Serre's $p$-adic weight space 
\[
\mathbb{Z}_p \times \mathbb{Z}/(p-1)\mathbb{Z} \simeq \varprojlim_{m \ge 1} \mathbb{Z}/(p-1)p^{m-1}\mathbb{Z},
\]
we may also regard 
$\{{\bf F}(P)\}$ as a usual $p$-adic analytic family 
parametrized by varying integers $\kappa \ge \kappa_0
$ with 
${\kappa \equiv {\kappa_0} \pmod{(p-1)p^{m-1}}}$ 
for some sufficiently large $m \ge 1$.

\begin{rem}
On purpose to 
construct a $\Lambda$-adic Siegel modular form ${\bf F}
$, we should take a $P_0=({\kappa_0}, \omega^{{\kappa_0}}) \in \mathfrak{X}_{\rm alg}(R)
$ having a smallest possible ${\kappa_0} \in \mathbb{Z}
$ such that 
${\bf F}(P_0)
$ coincides with an actual holomorphic Siegel modular form $F_{\kappa_0} \in \mathscr{M}_{{\kappa_0}}(\Gamma_0(p))^{(g)}$. 
For this reason, the condition ${\kappa_0} > g+1$ set out above, will be practically required 
in the subsequent arguments for the Duke-Imamo{\=g}lu lifting and the holomorphic Siegel Eisenstein series. Indeed, this is neither more nor less than the condition of holomorphy of the Siegel Eisenstein series of genus $g$.  
However, in the same context, it should better to assume a more general condition $\kappa_0 \ge g+1$, which is evident form the fact that the smallest possible weight for holomorphic Siegel modular forms of genus $g$ occurring in the de Rham cohomology is $g+1$. 
\end{rem}


\section{Cuspidal $\Lambda$-adic modular forms of genus 1}

In this section, we review Hida's construction of a cuspidal $\Lambda$-adic modular form of genus $1$ and of level $1$, and the $\Lambda$-adic Shintani lifting due to Stevens. 

\subsection{Hida's universal ordinary $p$-stabilized newforms}

For each integer $r \ge 1$, let $X_1(p^r)=\Gamma_1(p^r)\backslash \mathfrak{H}_1 \cup \mathbb{P}^{1}(\mathbb{Q})$ be the compactified modular curve, and $V_r=H^1(X_1(p^r),\mathbb{Z}_p)$ the simplicial cohomology group of $X_1(p^r)$ with values in $\mathbb{Z}_p$. It is well-known that $V_r$ is canonically isomorphic to the parabolic cohomology group $H_{\rm par}^1(\Gamma_1(p^r),\mathbb{Z}_p) \subseteq H^1(\Gamma_1(p^m),\mathbb{Z}_p)$, which is defined to be the image of the compact-support cohomology group 
under the natural map (cf. \cite{Shim94a}).  
We denote the {\it abstract $\Lambda$-adic Hecke algebra} of tame level $1$ by the free polynomial algebra 
\[
\mathbb{T}:=\Lambda_1[{\rm T}_m\,|\,1 \le m \in \mathbb{Z}]
\]
generated by ${\rm T}_m$ over $\Lambda_1$. Since $\mathbb{T} \simeq \mathbb{Z}_p[{\rm T}_m, \mathbb{Z}_p^{\times}]$, a natural action of $\mathbb{T}$ on $V_r$ is defined by regarding the generator ${\rm T}_m$ acts via the $m$-th Hecke correspondence and elements of $\mathbb{Z}_p^{\times}$ act via the usual Nebentypus actions. 
For each pair of positive integers $(r_1,r_2)$ with $r_1 \geq r_2 $, the natural inclusion $\Gamma_1(p^{r_1}) \hookrightarrow
 \Gamma_1(p^{r_2})$ induces the corestriction $V_{r_1} \to V_{r_2}$, which commutes with the actoin of $\mathbb{T}$. Hence we may consider the projective limit 
\[
V_{\infty} := \varprojlim_{r \ge 1}
 V_r
\]
with a $\mathbb{T}$-algebra structure. We denote by $V_{\infty}^{\rm ord}$ the direct factor of $V_{\infty}$
cut out by Hida's ordinary idempotent $e_{\rm ord}=\displaystyle\lim_{m \to \infty} {\rm T}_p^{m!}$, that is, 
\[
V_{\infty}^{\rm ord}=e_{\rm ord} \cdot V_{\infty}, 
\]
on which ${\rm T}_p$ acts invertibly. We note that $V_{\infty}^{\rm ord}$ is a $\Lambda$-algebra free of finite rank (cf. Theorem 3.1 in \cite{Hid86a}). 
Moreover, let $\mathcal{L}={\rm Frac}(\Lambda)$ be the fractional field of $\Lambda$, Hida constructed an idempotent $e_{\rm prim}$ in the image of $\mathbb{T}\otimes_{\Lambda} \mathcal{L}$ in ${\rm End}_{\mathcal{L}} (V_{\infty}^{\rm ord} \otimes_{\Lambda} \mathcal{L})$, which can be regarded as an analogue to the projection to the space of primitive Hecke eigenforms in Atkin-Lehner theory (cf. \cite{Hid86a}, pp.250, 252). 
Then we define the {\it universal ordinary parabolic cohomology group} of tame level $1$ by 
the $\mathbb{T}$-algebra
\[
\mathbb{V}^{\rm ord} := V_{\infty}^{\rm ord} \cap e_{\rm prim} (V_{\infty}^{\rm ord} \otimes_{\Lambda} \mathcal{L}), 
\]
which is a reflexive $\Lambda$-algebra of finite rank and is consequently locally free of finite rank over $\Lambda$. 
Then the {\it universal $p$-ordinary Hecke algebra} of tame level $1$ is defined to be 
the image $R^{\rm ord}$ of $\mathbb{T}$ in ${\rm End}_{\Lambda_1}(\mathbb{V}^{\rm ord})$ under the homomorphism 
\[
h: \mathbb{T} \longrightarrow
{\rm End}_{\Lambda_1}(\mathbb{V}^{\rm ord}). 
\]
We note that $R^{\rm ord}$ is naturally equipped with a formal $q$-expansion 
\begin{equation}
{\bf f}_{\rm ord}=\sum_{m=1}^{\infty} {\bf a}_m\, q^m \in R^{\rm ord}[[q]], \hspace*{5mm} 
{\bf a}_m=h({\rm T}_m), 
\end{equation}
which is called the {\it universal $p$-stabilized ordinary form} of tame level $1$. 

\vspace*{3mm}
Next, we introduce the global data to be interpolated by ${\bf f}_{\rm ord}$ as follows: 
\begin{defn}[ordinary $p$-stabilized newforms]
For given integers $\kappa \ge 2$ and $r \ge 1$, a Hecke eigenform $f_\kappa^{*} \in \mathscr{S}_{\kappa}(\Gamma_1(p^r))^{(1)}$ 
is called an {\it ordinary $p$-stabilized newform} if one of the following conditions holds true: 
\begin{itemize}
\item[(i)] $f_\kappa^*$ is a $p$-ordinary Hecke eigenform in $\mathscr{S}_{\kappa}^{\rm new}(\Gamma_1(p^r))^{(1)}$, where we denote by $\mathscr{S}_{\kappa}^{\rm new}(\Gamma_1(p^r))^{(1)}$ the subspace consisting of all newforms in $\mathscr{S}_{\kappa}(\Gamma_1(p^r))^{(1)}$. 

\item[(ii)] If $r=1$, then there exists a normalized ordinary Hecke eigenform $f_\kappa=\sum_{m=1}^{\infty} a_m(f_\kappa) q^m \in \mathscr{S}_{\kappa}({\rm SL}_2(\mathbb{Z}))
$ 
such that  
\[
f_\kappa^{*}(z)=f_\kappa(z) - \beta_p(f_{\kappa}) f_\kappa(pz)\,\,\,\, (z \in \mathfrak{H}_1), 
\]
where 
$\beta_p(f_\kappa)$ is the non-unit root of 
$X^2 - a_p(f_\kappa) X + p^{\kappa-1} =0$. 
\end{itemize}
\end{defn}

\begin{rem}
It follows from the definition that ordinary $p$-stabilized newforms 
are literally $p$-ordinary Hecke eigenforms. 
Indeed, the assertion is trivial in the case of (i).  
If $f_{\kappa}^{*} \in \mathscr{S}_{\kappa}(\Gamma_1(p))^{(1)}$ is taken as in (ii), then for each prime $l$, we have 
\[
a_l(f_{\kappa}^{*})=
\left\{
\begin{array}{cl}
a_l(f_{\kappa}) & \textrm{ if }l \ne p, \\[1mm]
\alpha_p(f_{\kappa}) & \textrm{ if }l = p,
\end{array}\right.
\]
where $\alpha_p(f_{\kappa})$ denotes the $p$-adic unit appearing in the factorization  
\[
X^2 - a_p(f_{\kappa}) X + p^{\kappa-1} =(X-\alpha_p(f_{\kappa}))(X-\beta_p(f_{\kappa})). 
\]
Hence we have \[L(s,\, f_{\kappa}^{*})=L^{(p)}(s,\,f_{\kappa}) \cdot(1-\alpha_p(f_{\kappa})p^{-s})^{-1},\] 
where $L^{(p)}(s,\,f_{\kappa})$ denotes Hecke's $L$-function of $f_{\kappa}$ with the 
Euler factor at $p$ removed. 
For a given $p$-ordinary Hecke eigenform in $\mathscr{S}_{\kappa}({\rm SL}_2(\mathbb{Z}))$, this type of $p$-adic normalization process selecting half the Euler factor at $p$ is called the {\it ordinary $p$-stabilization}. 
However, we should note that each 
ordinary $p$-stabilized newform $f_{\kappa}^{*} \in 
\mathscr{S}_{\kappa}(\Gamma_1(p))^{(1)}$ 
is actually an oldform except for $\kappa=2$. 
Indeed, if $f_{\kappa}^{*} \in \mathscr{S}_{\kappa}^{\rm new}(\Gamma_1(p))^{(1)}$, 
then we have 
\[
|a_p(f_{\kappa}^{*})|
=p^{\kappa/2-1}
\]
(cf. Theorem 4.6.17 (ii) in \cite{Miy06}). 
If $\kappa >2$, this contradicts the assumption that $f_{\kappa}^{*}$ is ordinary at $p$. 
Hence we summarize that each ordinary $p$-stabilized newform 
is in fact a $p$-ordinary Hecke eigenform occurring in either 
$\mathscr{S}_{2}^{\rm new}(\Gamma_1(p))^{(1)}$, 
$\mathscr{S}_{\kappa}^{\rm old}(\Gamma_1(p))^{(1)}:= \mathscr{S}_{\kappa}(\Gamma_1(p))^{(1)} -  \mathscr{S}_{\kappa}^{\rm new}(\Gamma_1(p))^{(1)}$ with $\kappa > 2$, or $\mathscr{S}_{\kappa}^{\rm new}(\Gamma_1(p^r))^{(1)}$ with some $\kappa \ge 2$ and $r > 1$. 
\end{rem}

Then the following theorem has been established by Hida: 

\begin{hidathm}[cf. 
Theorem 2.6 in \cite{G-S93}]
Let $r$ be a fixed positive integer, and ${\bf f}_{\rm ord}=\sum_{m = 1}^{\infty} \mathbf{a}_m q^m \in R^{\rm ord}[[q]]$ 
the universal ordinary $p$-stabilized form of tame level $1$ introduced above. 
Then for each $P \in \mathfrak{X}_{\rm alg}(R^{\rm ord})$, the specialization 
\[
{{\bf f}_{\rm ord}(P)=
\sum_{m = 1}^{\infty} \mathbf{a}_m(P) 
q^m \in \overline{\mathbb{Q}}_p[[q]]}
\] 
induces a one-to-one correspondence 
\begin{eqnarray*}
\lefteqn{
\left\{
P=(\kappa, \varepsilon) \in 
\mathfrak{X}_{\rm alg}(R^{\rm ord}) \left|\, 
2 \le \kappa \in \mathbb{Z},\,\,
\varepsilon : \mathbb{Z}_p^{\times} \to \overline{\mathbb{Q}}_p^{\times}\,(\textrm{finite character})
\right.\right\}
} \\
&\stackrel{1:1}{\longleftrightarrow}&
\left\{
f_{\kappa}^{*} \in \mathscr{S}_{\kappa}(\Gamma_0(p^r),\varepsilon \omega^{-\kappa}
)^{(1)} \left|
\begin{array}{c}
\textrm{ordinary }p\textrm{-stabilized newform} \\
\textrm{of tame level }1
\end{array}\right.
\right\}.
\end{eqnarray*}
\end{hidathm}

\begin{rem}
In most generality, for each positive integer $N$ prime to $p$,  
Hida constructed the universal ordinary $p$-stabilized newform of tame level $N$ whose specialization 
gives holomorphic cusp forms on $\Gamma_0(N p^r)$ with $r \ge 1$.  
\end{rem}

By applying Theorem II 
for $r=1$,  
each $P =(2k, \omega^{2k}) \in \mathfrak{X}_{\rm alg}(R^{\rm ord})$ 
corresponds to an ordinary $p$-stabilized newform ${\bf f}_{\rm ord}(P)=
f_{2k}^{*} \in \mathscr{S}_{2k}(\Gamma_0(p))^{(1)}$ associated with 
a $p$-ordinary normalized Hecke eigenform $f_{2k} \in \mathscr{S}_{2k}({\rm SL}_2(\mathbb{Z}))$ via the ordinary $p$-stabilization. However, 
we note that $\dim_{\mathbb{C}} \mathscr{S}_{2k}({\rm SL}_2(\mathbb{Z}))=0$ for $k <6$, and hence 
${\bf f}_{\rm ord}(P)$ vanishes identically at $\{P =(2k, \omega^{2k}) \in \mathfrak{X}_{\rm alg}(R^{\rm ord})\,|\,1 < k< 6\}$. 
Therefore, for a fixed $P_0=(2k_0, \omega^{2k_0})\in \mathfrak{X}_{\rm alg}(R^{\rm ord})$ with $k_0 \ge 6$, we may regard ${\bf f}_{\rm ord} \in R^{\rm ord}[[q]]$ as a $\Lambda$-adic cusp form of genus $1$, 
and 
we consequently obtain a $p$-adic analytic family of ordinary $p$-stabilized newforms
$\{{\bf f}_{\rm ord}(P)=f_{2k}^{*}\}$ 
parametrized by 
$P =(2k, \omega^{2k})\in \mathfrak{X}_{\rm alg}(R^{\rm ord})$ with $k_0 \le k \in \mathbb{Z}$.  

\vspace*{3mm}
In our setting, the choice of $P_0 = (2k_0, \omega^{2k_0})$
having the smallest possible weight $
2k_0$ is obviously taken as $k_0=6$, that is, $P_0$ corresponds to Ramanujan's $\Delta$-function \[
f_{12}=q\prod_{m=1}^{\infty}(1 - q^m)^{24}=\sum_{m=1}^{\infty} \tau(m) q^m
 \in \mathscr{S}_{12}({\rm SL}_2(\mathbb{Z})). 
 \]
In addition, we may choose any analytic neighborhood $\mathfrak{U}_0$ of such $P_0$ in $\mathfrak{X}(R^{\rm ord})$. 
Since we will apply some 
lifting for 
${\bf f}_{\rm ord}$ in the sequel, the choices of $k_0$, 
$P_0$ and $\mathfrak{U}_0$ may vary depending on the intended use. 
For readers' convenience, 
we present a list of ordinary primes with respect to the unique normalized Hecke eigenforms $f_{2k_0} \in \mathscr{S}_{2k_0}({\rm SL}_2(\mathbb{Z}))$ with $k_0 \in \{6,\,8,\,9,\,10,\,11,\,13\}$, that is, rational primes at which $f_{2k_0}$ is ordinary:  
\begin{table}[htbp]
\begin{tabular}{|c|c|}
\hline
$k_0$  & $p$ \\
\hline
6 & $11\le p \le 2399
,\,2417 \le p \le 19597$ \\
\hline
8 & $17 \le p \le 53
, \, 61 \le p \le 15269
, \, 15277 \le p \le 19597$ \\
\hline
9 & $17 \le p \le 14879$ \\
\hline
10 & $19 \le p \le 3361
,\, 3373 \le p \le 9973$ \\
\hline
11 & $p=11,\, 23 \le p \le 9973$ \\
\hline
13 & $29 \le p \le 9973$ \\
\hline 
\end{tabular}
\caption{Ordinary primes for $f_{2k_0}$}
\end{table}

\vspace*{-4mm}
For the smallest ordinary prime $p=11$ with respect to 
$f_{12}$, we give a numerical example of another components of the Hida family: 

\begin{ex}
Since $6+(11-1)=16$, we focus on the 
2-dimensional space $\mathscr{S}_{32}({\rm SL}_2(\mathbb{Z}))$.  
Then we may take a normalized Hecke eigenform $f_{32} \in \mathscr{S}_{32}({\rm SL}_2(\mathbb{Z}))$ 
determined uniquely up to Galois conjugation
such that
\begin{eqnarray*}
f_{32} 
&=&q + x q^2 + (432x + 50220)q^3 + (39960x + 87866368)q^4 \\
&& - (1418560x- 18647219790)q^5 + (17312940x + 965671206912)q^6 \\
&& -(71928864x - 16565902491320)q^7 
-(462815680x - 89324586639360)
  q^8  \\
&& + (7500885120x - 200500912849563) q^9 \\
&&- (38038437810x + 3170978118696960)q^{10} \\
&&+ (29000909200x - 4470615038375388)q^{11} +
\cdots 
\in K_{32}[[q]],
\end{eqnarray*}
where $K_{32}$ denotes the real quadratic field \[
\mathbb{Q}[x]/(x^2 - 39960x - 2235350016). 
\] 
We easily check that the norm of the difference $a_{11}(f_{12}) - a_{11}(f_{32})$ is factored into 
\[
2^8 \cdot 3^3 \cdot 5^4 \cdot 11 \cdot 368789 \cdot 99988481 \cdot 7376353157. 
\]
Therefore we obtain a congruence between $f_{12}$ and $f_{32}$ modulo a prime ideal of the ring $\mathcal{O}_{K_{32}}$ of integers in $K_{32}$ 
lying over $11$, which implies that their ordinary $11$-stabilizations 
$f_{12}^{*}$ and $f_{32}^{*}$ 
reside both in the Hida family for $p=11$.  
\end{ex}

\vspace*{4mm}
\subsection{$\Lambda$-adic Shintani lifting}
As mentioned in the previous \S\S2.1, we reveiw the construction of 
an inverse correspondence of the 
Shimura correspondence in the sense of Shintani \cite{Shin75} and Kohnen \cite{Koh85}. Moreover, we introduce a similar lifting for the universal ordinary $p$-stabilized newform ${\bf f}_{\rm ord}$, which was constructed by Stevens \cite{Ste94}. \\

For simplicity, suppose that $N \ge 1$ is odd squarefree and $k \ge 2$. 
Let $\mathfrak{D}$ be an integer with $\mathfrak{D} \equiv 0,\,1 \pmod{4}$ 
and $(-1)^k \mathfrak{D}>0$.  
We denote by $\mathcal{L}(\mathfrak{D})$ the set of all primitive matrices $Q \in {\rm Sym}_2^{*}(\mathbb{Z})$ 
with discriminant $
-\det(2Q)=\mathfrak{D}$. 
We may naturally identify each element $Q=\left(\begin{smallmatrix}
a & b/2 \\
b/2 & c
\end{smallmatrix}\right) \in \mathcal{L}(\mathfrak{D})$
with an integral binary quadratic form $
Q(x,y)=ax^2 +bxy+cy^2$ with $\gcd(a,b,c)=1$. 
For simplicity, 
we write $Q=[a,b,c]$ instead of $\left(\begin{smallmatrix}
a & b/2 \\
b/2 & c
\end{smallmatrix}\right)$ in the sequel. 
We also denote by $\mathcal{L}_N(\mathfrak{D})$ the subset of $\mathcal{L}(\mathfrak{D})$ consisting of all elements $[a,b,c]$ with $a \equiv 0 \pmod{N}$. We easily see that if $\mathfrak{D} \equiv 0 \pmod{N}$, 
then 
\[
\mathcal{L}_N(\mathfrak{D})=\{\,[a,b,c] \in \mathcal{L}(\mathfrak{D})\,|\, a \equiv b \equiv 0 \!\!\!\pmod{N}\,\}. 
\]
We note that the congruence subgroup $\Gamma_0(N) \subseteq {\rm SL}_2(\mathbb{Z})$ acts on $\mathcal{L}_N(\mathfrak{D})$ as 
\[
\begin{array}{ccl}
\mathcal{L}_N(\mathfrak{D}) 
\times 
\Gamma_0(N)
& \longrightarrow & \mathcal{L}_N(\mathfrak{D}) \\[2mm]
(Q, M) & \longmapsto & Q \circ M:=
{}^t M Q M,  
\end{array}
\]
and we easily see that $
\mathcal{L}_N(\mathfrak{D})/\Gamma_0(N)$ is finite. 
For each $Q=[a,b,c] \in \mathcal{L}_N(\mathfrak{D})$, we associate it with a geodesic cycle $C_Q$ 
in $\Gamma_0(N)\backslash \mathfrak{H}_1$ that is defined as the image of the semicircle $$az^2+b{\rm Re}(z)+c=0$$ oriented either from left to right (resp. from right to left) if $a>0$ (resp. $a<0$) or from $-c/b$ to $\sqrt{-1} \infty$ according as $a\ne0$ or $a=0$. 
Then for each $f \in \mathscr{S}_{2k}(\Gamma_0(N))$, we define a cycle integral associated with $f$ by
\[
r_{Q}(f):=\displaystyle\int_{C_Q
} f(z) 
Q(z,1)^{k-1} dz. 
\]
Then the following theorem is given by Shintani and Kohnen: 

\begin{shintanithm}[cf. Theorem 2 in \cite{Koh81}]
Let $\mathfrak{d}$ be a fixed fundamental discriminant with $(-1)^k \mathfrak{d} >0$. 
For each $f \in \mathscr{S}_{2k}(\Gamma_0(N))$, put  
\begin{eqnarray*}
\lefteqn{
\vartheta_{\mathfrak{d}}(f):=\sum_{\scriptstyle m \ge 1, \atop {\scriptstyle (-1)^k m \equiv 0,1 \hspace{-2mm}\pmod{4}}} 
\!\!\left\{
\sum_{0<d | N} \mu(d) \left({\,\mathfrak{d}\,\over d}\right) d^{k-1} \right.} \\
&&\hspace*{35mm}\times\left.
\sum_{Q \in \mathcal{L}_{(Nd)}(\mathfrak{d}md^2)/\Gamma_0(Nd)} \chi_{\mathfrak{d}}(Q)\, r_{Q}(f)
\right\} q^m,
\end{eqnarray*}
where 
$\chi_{\mathfrak{d}}
$ denotes the generalized genus character associated with ${\mathfrak{d}_0}$ {\rm (cf.} \cite{GKZ87}{\rm ).} 
Then we have $\vartheta_{\mathfrak{d}}(f) \in \mathscr{S}_{k+1/2}^{+}(\Gamma_0(4N))^{(1)}$. Moreover, 
the mapping 
\[
\vartheta_{\mathfrak{d}} : \mathscr{S}_{2k}(\Gamma_0(N))^{(1)} \longrightarrow \mathscr{S}_{k+1/2}^{+}(\Gamma_0(4N))^{(1)} 
\]
is Hecke equivariant 
in the sense of the Shimura correspondence.  
\end{shintanithm}

\vspace*{2mm}
This type of lifting from integral weight to half-integral weight was firstly introduced by Shintani \cite{Shin75}, and afterwards was reformulated by Kohnen \cite{Koh81}. According to the custom, we refer to $\vartheta_{\mathfrak{d}}$ as the {\it $\mathfrak{d}$-th Shintani lifting}. \\

For a given normalized Hecke eigenform $f \in \mathscr{S}_{2k}(\Gamma_0(N))$, we note that all of 
the Shintani lifting $\vartheta_{\mathfrak{d}}(f)$ 
give rise to Hecke eigenforms in $\mathscr{S}_{k+1/2}^{+}(\Gamma_0(4N))^{(1)}$ corresponding to $f$ via the Shimura correspondence, however they differ from each other by the normalization 
of the Fourier coefficients depending on $\mathfrak{d}$. 
Indeed, suppose for simplicity that $N=1$, $f \in \mathscr{S}_{2k}({\rm SL}_2(\mathbb{Z}))$ is a normalized Hecke eigenform and $h \in \mathscr{S}_{k+1/2}^{+}(\Gamma_0(4))^{(1)}$ a corresponding Hecke eigenform via the Shimura correspondence as in (2). Then 
for each integer $m \ge 1$ with $(-1)^k m \equiv 0,\,1 \pmod{4}$, 
we have 
\begin{eqnarray}
{c_{|\mathfrak{d}|}(h) c_m(h) \over \|h\|^2}&=&{(-1)^{[k/2]} 2^k \over \|f\|^2} 
\!\!\sum_{Q \in \mathcal{L}(\mathfrak{d}m)/{\rm SL}_2(\mathbb{Z})} \!\!\chi_{\mathfrak{d}}(Q)\, r_{Q}(f) \\
&=& {(-1)^{[k/2]} 2^k \over \|f\|^2} c_m(\vartheta_{\mathfrak{d}}(f)) \nonumber
\end{eqnarray}
(cf. Theorem 3 in \cite{Koh85}). 
Namely, the choice of $\mathfrak{d}$ determines the normalization datum. 

\begin{rem}
We note that the $\mathfrak{d}$-th Shintani lifting $\vartheta_{\mathfrak{d}}(f)$ admits a nice algebraic property. Indeed, by the equations (8) and (11), we have 
\[
c_{|\mathfrak{d}|}(\vartheta_{\mathfrak{d}}(f))=(-1)^{[k/2]}{(k-1)! \over (2\pi)^{k}} \cdot |\mathfrak{d}|^{k-1/2} L_{\mathfrak{d}}(k,\,f). 
\]
On the other hand, by Manin \cite{Man73} and Shimura \cite{Shim82}, we may associate $f$ with two complex periods $\Omega^{+}$ and $\Omega^{-}$ such that for each critical point $s \in \mathbb{Z}$ with $0 < s < 2k$, the special value 
$\pi^{-s}L_{\mathfrak{d}}(s,\,f)/\Omega^{\epsilon}$ resides in the field 
$
K_{f}(\sqrt{|\mathfrak{d}|}),$
where $\epsilon \in \{\pm\}$ is the signature of 
$(-1)^k \left({\mathfrak{d} \over -1}\right)$. 
Therefore we obtain 
${c_{|\mathfrak{d}|}(\vartheta_{\mathfrak{d}}(f))/\Omega^{\epsilon
} \in 
K_{f}(\sqrt{|\mathfrak{d}|})}$. Moreover, let $\mathcal{O}_f$ be 
the ring of integers in $K_f
$. Then by combining Kohnen's theory and a result of Stevens (cf. Proposition 2.3.1 in \cite{Ste94}), 
we have 
$${1\over \Omega^{\epsilon}}\,\vartheta_{\mathfrak{d}}(f) \in \mathcal{O}_f[[q]]$$
after taking a suitable normalization. 
\end{rem}

Now, let us consider a $\Lambda$-adic analogue of the Shintani lifting for the universal ordinary $p$-stabilized newform ${\bf f}_{\rm ord} \in R^{\rm \,ord}[[q]]$ (cf. (10) in \S\S 3.1): 
We define the {\it metaplectic double covering} of the universal ordinary Hecke algebra $R^{\rm \,ord}$ 
by 
\[
\widetilde{R}^{\rm \,ord} := R^{\rm \,ord} \otimes_{\Lambda_1,\sigma} \Lambda_1, 
\]
where 
the tensor product is taken with respect to the ring homomorphism $\sigma : \Lambda_1 \to \Lambda_1$ corresponding to the group homomorphism $y \mapsto y^2$ on $\mathbb{Z}_p^{\times}$. We note that $\widetilde{R}^{\rm \,ord}$ has a natural $\Lambda_1$-algebra structure induced from the homomorphism $\lambda \mapsto 1 \otimes \lambda$ on $\Lambda_1$. Therefore we may define the associated $\Lambda$-adic analytic space $\mathfrak{X}(\widetilde{R}^{\rm ord})$ and its subset $\mathfrak{X}_{\rm alg}(\widetilde{R}^{\rm ord})$ of arithmetic points 
as well as $R^{\rm ord}$. 
However, we should mention that the ring homomorphism 
\[
\begin{array}{lcc}
R^{\rm ord}& \longrightarrow &\widetilde{R}^{\rm \,ord} \\
\,\mathbf{a} &\longmapsto &\mathbf{a} \otimes 1
\end{array}
\]
is not a $\Lambda_1$-algebra homomorphism. This causes the fact that the mapping induced by pullback on $\Lambda$-adic analytic spaces 
$\mathfrak{X}(\widetilde{R}^{\rm ord}) \to \mathfrak{X}(R^{\rm ord})$ 
does not preserve the signatures of arithmetic points. Indeed, we easily see that if $\widetilde{P}=(\kappa,\varepsilon) \in\mathfrak{X}_{\rm alg}(\widetilde{R}^{\rm ord})$ lies over $P \in \mathfrak{X}_{\rm alg}(R^{\rm ord})$, then $P=(2\kappa, \varepsilon^2)$. 

\vspace*{3mm}
Then the following theorem is a refinement of Stevens' $\Lambda$-adic Shintani lifting for ${\bf f}_{\rm ord} \in R^{\rm \,ord}[[q]]$: 

\begin{stevensthm}[cf. Theorem 3 in \cite{Ste94}]
For a fixed $P_0=(2k_0, \omega^{2k_0 }) \in \mathfrak{X}_{\rm alg}(R^{\rm ord})$ with $k_0 >1$, let 
$\mathfrak{d}_0$ be a fundamental discriminant with $(-1)^{k_0} \mathfrak{d}_0>0$ and $\mathfrak{d}_0 \equiv 0 \pmod{p}$. 
Then there exist 
a formal $q$-expansion \[
\Theta_{\mathfrak{d}_0}=\sum_{m\ge1} {\bf b}(\mathfrak{d}_0;\,m) \,q^m \in \widetilde{R}^{\rm \,ord}[[q]]\] 
and a choice of $p$-adic periods $\Omega_P \in \overline{\mathbb{Q}}_p$ for $P \in \mathfrak{X}_{\rm alg}(R^{\rm ord})$ {\rm (cf. \cite{G-S93})} satisfying the following: 
\begin{itemize}
\item[(i)] $\Omega_{P_0} \ne 0$. 

\item[(ii)] For each $\widetilde{P}=(k,\omega^{k}) \in \mathfrak{X}_{\rm alg}(\widetilde{R}^{\rm \,ord})$, the specialization 
{\[
\Theta_{\mathfrak{d}_0}(\widetilde{P})= \sum_{m \ge 1} {\bf b}(\mathfrak{d}_0;\,m) (\widetilde{P})\,q^m  \in \overline{\mathbb{Q}}_p[[q]]
\]}
\hspace*{-0.5mm}gives rise to a holomorphic cusp form in $\mathscr{S}_{k+1/2}^{+}(\Gamma_0(4p))^{(1)}$. \\
In particular, there exists an analytic neighborhood $\mathfrak{U}_0$ of $P_0$ \\
such that for each $\widetilde{P}$ lying over $P=(2k, \omega^{2k})
 \in \mathfrak{U}_0$, we have 
{\[
\Theta_{\mathfrak{d}_0}(\widetilde{P})={\Omega_{P}
\over \Omega^{\epsilon}(P)
}\, \vartheta_{\mathfrak{d}_0}({\bf f}_{\rm ord}(P)), 
\]}
where $\Omega^{\epsilon}(P)$ denotes 
the complex periods of ${\bf f}_{\rm ord}(P)$ with \\
\,\,signature $\epsilon \in \{\pm \}$. 
\end{itemize}
\end{stevensthm}

We note that a non-vanishing property of $\Theta_{\mathfrak{d}_0}$ is naturally induced from the one of the classical Shintani lifting $\vartheta_{\mathfrak{d}_0}$. 
Indeed, if $\vartheta_{\mathfrak{d}_0}({\bf f}_{\rm ord}(P_0))$ is non-zero, then by virtue of the property (i), $\Theta_{\mathfrak{d}_0}$ does not vanish on an analytic neighborhood $\mathfrak{U}_0$ of $P_0$. 
Hence a suitable choice of $(\mathfrak{d}_0, \mathfrak{U}_0)$ yields a $p$-adic analytic family of non-zero half-integral weight forms $\{{\Omega_{P}
\over \Omega^{\epsilon}(P)
}\,\vartheta_{\mathfrak{d}_0}({\bf f}_{\rm ord}(P))\}$ parametrized by varying arithmetic points $P=(2k, \omega^{2k}) \in \mathfrak{U}_0$ with $k \ge k_0$. 
For further details on the non-vanishing properties of the classical Shintani lifting and its generalizations, see \cite{Wal80, Wal91}.  

\vspace*{3mm}
For each $P=(2k, \omega^{2k}) \in \mathfrak{U}_0$ with $k \ge k_0$, 
let $f_{2k} \in \mathscr{S}_{2k}({\rm SL}_2(\mathbb{Z}))$ be 
a $p$-ordinary normalized Hecke eigenform corresponding to ${\bf f}_{\rm ord}(P)
 \in \mathscr{S}_{2k}(\Gamma_0(p))^{(1)}$ (i.e. ${\bf f}_{\rm ord}(P)=f_{2k}^{*}$). 
Then by Theorem IV, we may also resolve the $p$-adic interpolation problem for some Fourier coefficients of the classical Shintani lifting $\vartheta_{\mathfrak{d}_0}(f_{2k}) \in \mathscr{S}_{k+1/2}^{+}(\Gamma_0(4))^{(1)}$. 

\vspace*{3mm}
For such occasions, we prepare the following: 

\begin{lem}
Suppose that $\mathfrak{D} \equiv 0 \pmod{p}$. Then we have 
\begin{itemize}
\item[(i)] For each $Q 
\in \mathcal{L}(\mathfrak{D})
$, there exists $Q'
 \in \mathcal{L}_p(\mathfrak{D})$ such that 
 \[
 Q' \equiv Q \pmod{{\rm SL}_2(\mathbb{Z})}
 . \]

\item[(ii)] If $\mathfrak{D} \equiv 0 \pmod{p^2}$ and $\left({\mathfrak{D}/p^2 \over p}\right)=1$, then 
for each $Q 
\in \mathcal{L}_p(\mathfrak{D})
$, there exists $[a,b,c]
 \in \mathcal{L}_p(\mathfrak{D})$ such that $a \equiv 0 \pmod{p^3}$ and  
 \[
Q \equiv [a,b,c] \pmod{\Gamma_0(p)}
 . \]

\item[(iii)] 
The identification  
$[a,b,c] 
\,\,\mathrm{mod}\,\, \Gamma_0(p) 
\mapsto 
[a,b,c]$
$\mathrm{mod}\,\, {\rm SL}_2(\mathbb{Z})
$ induces a one-to-one correspondence between 
$\mathcal{L}_p(\mathfrak{D})/\Gamma_0(p)$ and $\mathcal{L}_p(\mathfrak{D})/{\rm SL}_2(\mathbb{Z})$. 

\item[(iv)] If $\mathfrak{D} \not\equiv 0 \pmod{p^2}$, then the mapping 
$[a,b,c] 
\,\,\mathrm{mod}\,\, \Gamma_0(p) 
 \mapsto [p^{-1}a,b,pc] $ 
 $\mathrm{mod}\,\, {\rm SL}_2(\mathbb{Z})
$ induces a one-to-one correspondence between 
$\mathcal{L}_p(\mathfrak{D})/\Gamma_0(p)$ and $\mathcal{L}(\mathfrak{D})/{\rm SL}_2(\mathbb{Z})$. 

\end{itemize}
\end{lem}

\begin{proof}
The assertions (i), (ii), (iii) and (iv) have appeared respectively as Lemmas 1, 3, 2 of \cite{Gue00} and Proposition in \S I.1 of \cite{GKZ87}. For the readers' convenience, we present all of their proofs. 
For each $[a,b,c] \in \mathcal{L}(\mathfrak{D})$ with $b \not\equiv 0 \pmod{p}$, the assumption $\mathfrak{D} \equiv 0 \pmod{p}$ yields 
$a \not\equiv 0 \pmod{p}$. Then we put $\beta \equiv -b/(2a) \pmod{p}$ and  
\[
{
[a',b',c']:=[a,b,c]\circ 
\left(
\begin{array}{cc}
1 & \beta \\
0 & 1
\end{array}\right) \equiv [a,b,c] \pmod{{\rm SL}_2(\mathbb{Z})}. 
}
\]
We note that $b' = 2a \beta +b \equiv 0 \pmod{p}$. 
Hence we may assume that $b \equiv 0 \pmod{p}$. 
If $[a,b,c] \in \mathcal{L}(\mathfrak{D})$ satisfies $a \not\equiv 0 \pmod{p}$, then we easily see that $c \equiv 0 \pmod{p}$. Hence we have 
\[
{
[a,b,c] \equiv [a,b,c]\circ 
\left(
\begin{array}{cc}
p & p-1 \\
1 & 1
\end{array}\right) \in \mathcal{L}_p(\mathfrak{D}) \pmod{{\rm SL}_2(\mathbb{Z})}
}
\]
and we obtain the assertion (i). 
For each $Q=[a,b,c] \in \mathcal{L}_p(\mathfrak{D})$ with $\mathfrak{D}\equiv 0 \pmod{p^2}$, we easily see that $a \equiv 0 \pmod{p^2}$. 
If $Q'=[a',b',c']\in \mathcal{L}_p(\mathfrak{D})$ satisfies $Q'
=
Q
 \circ \left(\begin{smallmatrix}
\alpha & \beta \\
\gamma & \delta
\end{smallmatrix}\right)$ 
with some $\left(\begin{smallmatrix}
\alpha & \beta \\
\gamma & \delta
\end{smallmatrix}\right) \in \Gamma_0(p)
$,  
then the condition $\left({\mathfrak{D}/p^2 \over p}\right)=1$ yields that the quadratic form $[p^{-2}a, p^{-1}b, c] \in \mathcal{L}_p(\mathfrak{D}/p^2)$ satisfies 
\[
[p^{-2}a, p^{-1}b, c](\alpha, p^{-1}\gamma ) = (p^{-2}a) \alpha^2 + (p^{-1}b) \alpha (p^{-1}\gamma) + c(p^{-1}\gamma)^2 \equiv 0 \pmod{p}. 
\]
This equation implies the fact that $Q'$ satisfies $p^{-2}a' \equiv 0 \pmod{p}$, and hence we obtain the assertion (ii). 
For the proof of the assertion (iii), it suffices to show the injectivity of the mapping 
$[a,b,c] 
\,\,\mathrm{mod}\,\, \Gamma_0(p) 
\mapsto 
[a,b,c] \,\,\mathrm{mod}\,\, {\rm SL}_2(\mathbb{Z})
$. 
Indeed, for $[a,b,c]$, $[a',b',c'] \in \mathcal{L}_p(\mathfrak{D})$, if 
there exists $\left(\begin{smallmatrix}
\alpha & \beta \\
\gamma & \delta
\end{smallmatrix}\right) \in {\rm SL}_2(\mathbb{Z})
$ such that 
\[
[a',b', c'] =[a,b, c] \circ \left(
\begin{array}{cc}
\alpha & \beta \\
\gamma & \delta
\end{array}\right),\]
then we have $a'=a\alpha^2 +b \alpha \gamma + c \gamma^2$. Since $a, a', b \equiv 0 \pmod{p}$ and $\gcd(c,p)=1$, we obtain 
$\gamma \equiv 0 \pmod{p}$. Namely, $[a,b,c] \equiv [a',b',c'] \pmod{\Gamma_0(p)}$. 
In order to prove the assertion (iv), it also suffices to show that the injectivity of the mapping  
$[a,b,c] \,\,\mathrm{mod}\,\, \Gamma_0(p) 
\mapsto [p^{-1}a,b,pc] \,\, \mathrm{mod}\,\, {\rm SL}_2(\mathbb{Z})
$. 
Indeed, suppose that $[a,b,c]$, $[a',b',c'] \in \mathcal{L}_p(\mathfrak{D})$ 
satisfy 
\[
[p^{-1}a',b', pc'] =[p^{-1}a,b, pc] \circ \left(
\begin{array}{cc}
\alpha & \beta \\
\gamma & \delta
\end{array}\right)\]
with $\left(\begin{smallmatrix}
\alpha & \beta \\
\gamma & \delta
\end{smallmatrix}\right) \in {\rm SL}_2(\mathbb{Z})
$. The assumption $\mathfrak{D} \not\equiv 0 \pmod{p^2}$ yields $a \not\equiv 0 \pmod{p^2}$. 
Since $pc'=p^{-1}a \beta^2 + b\beta \delta + p c \delta^2$, 
we have $\beta \equiv 0 \pmod{p}$. Then we have 
\[
[a',b', c'] =[a,b, c]\circ \left(
\begin{array}{cc}
\alpha & p^{-1}\beta \\
p\gamma & \delta
\end{array}\right), 
\]
and hence $[a',b',c'] \equiv [a,b,c] \pmod{\Gamma_0(p)}$. 
We complete the proof. 
\end{proof}

As a consequence of Theorem IV, we have the following: 

\begin{prop} 
Under the same notation and assumptions as above,  
let
$\mathfrak{d}_0$ be a fixed fundamental discriminant with $(-1)^{k_0}\mathfrak{d}_0 >0$, $\mathfrak{d}_0 \equiv 0 \pmod{p}$ and $c_{|\mathfrak{d}_0|}(\vartheta_{\mathfrak{d}_0}(f_{2k_0})) \ne 0$ {\rm (cf. Lemma 2.4)}. 
Then for each integer $k \ge k_0$ and each fundamental discriminant $\mathfrak{d}$ with $(-1)^k \mathfrak{d}>0$, 
there exist an analytic neighborhood $\mathfrak{U}_0$ of $P_0=(2k_0, \omega^{2k_0}) \in \mathfrak{X}_{\rm alg}(R^{\rm ord})$ and an element ${\bf c}(\mathfrak{d}_0;\,{|\mathfrak{d}|}) \in (\widetilde{R}^{\rm ord})_{(\widetilde{P}_0)}$ such that 
\[
{\bf c}(\mathfrak{d}_0;\,{|\mathfrak{d}|})(\widetilde{P})={\Omega_{P} \over \Omega^{\epsilon}(P)}\,\left(1- \left({\,\mathfrak{d}\, \over p}\right) \beta_p(f_{2k})\, p^{-k}\right) c_{|\mathfrak{d}|}(
\vartheta_{\mathfrak{d}_0}(f_{2k}))
\]
for each $\widetilde{P} =(k,\omega^{k})\in \mathfrak{X}_{\rm alg}(\widetilde{R}^{\rm ord})$ lying over $P=(2k,\omega^{2k}) \in \mathfrak{U}_0$. 
\end{prop}

\begin{proof}
The following proof is essentially the same as the one of the main theorem in \cite{Gue00}: 
By virtue of Theorem IV, there exists an element ${\bf b}(\mathfrak{d}_0;\,{|\mathfrak{d}|}) \in \widetilde{R}^{\rm ord}$
such that for each $\widetilde{P} \in \mathfrak{X}_{\rm alg}(\widetilde{R}^{\rm ord})$ 
lying over $P=(2k,\omega^{2k}) \in \mathfrak{U}_0$, 
\begin{eqnarray*}
{\bf b}(|\mathfrak{d}_0|;\,\mathfrak{d}) (\widetilde{P})&=&{\Omega_{P} \over \Omega^{\epsilon}(P)}\,c_{|\mathfrak{d}|}(\vartheta_{\mathfrak{d}_0}(f_{2k}^{*})) \\
&=&{\Omega_{P} \over \Omega^{\epsilon}(P)}\,\sum_{Q \in \mathcal{L}_{p}(\mathfrak{d}_0|\mathfrak{d}|)/\Gamma_0(p)} \chi_{\mathfrak{d}_0}(Q)\, r_{Q}(f_{2k}^{*}). 
\end{eqnarray*} 
(I) Suppose that $\mathfrak{d} \not\equiv 0 \pmod{p}$. Then by (i), (iii) of Lemma 3.6, we have 
\begin{eqnarray*}
&& \sum_{Q \in \mathcal{L}_{p}(\mathfrak{d}_0|\mathfrak{d}|)/\Gamma_0(p)} \chi_{\mathfrak{d}_0}(Q)\, r_{Q}(f_{2k}^{*})
= \sum_{Q \in \mathcal{L}(\mathfrak{d}_0|\mathfrak{d}|)/{\rm SL}_2(\mathbb{Z})} \chi_{\mathfrak{d}_0}(Q)\, r_{Q}(f_{2k}) \hspace*{25mm} \\
&& \hspace*{0mm}- \beta_p(f_{2k}) \sum_{[a,b,c] \in \mathcal{L}_{p}(\mathfrak{d}_0|\mathfrak{d}|)/\Gamma_0(p)} \chi_{\mathfrak{d}_0}([a,b,c])\, 
\displaystyle\int_{C_{[a,b,c]}
} f_{2k}(pz)
(az^2+bz +c)^{k-1} dz.  
\end{eqnarray*}
Here we easily see that 
\[
\chi_{\mathfrak{d}_0}([a,b,c])=\left({\,\mathfrak{d}\,\over p}\right)\chi_{\mathfrak{d}_0}([p^{-1}a,b,pc])
\]
for each $[a,b,c] \in \mathcal{L}_{p}(\mathfrak{d}_0|\mathfrak{d}|)$, 
and hence it follows from (iv) of Lemma 3.6 that 
\begin{eqnarray*}
\lefteqn{\sum_{[a,b,c] \in \mathcal{L}_{p}(\mathfrak{d}_0|\mathfrak{d}|)/\Gamma_0(p)} \chi_{\mathfrak{d}_0}([a,b,c])\, 
\displaystyle\int_{C_{[a,b,c]}
} f_{2k}(pz)
(az^2+bz +c)^{k-1} dz} \\
&=&
\left({\,\mathfrak{d}\,\over p}\right)
\sum_{[a,b,c] \in \mathcal{L}(\mathfrak{d}_0|\mathfrak{d}|)/{\rm SL}_2(\mathbb{Z})} \chi_{\mathfrak{d}_0}([a,b,c])\, 
\\
&& \hspace*{10mm} \times \displaystyle\int_{C_{[a,b,c]}
} f_{2k}(pz)
\{a(pz)^2+b(pz) +c\}^{k-1} \cdot p^{1-k} dz \\
&=&
\left({\,\mathfrak{d}\,\over p}\right) p^{-k}
\sum_{Q=[a,b,c] \in \mathcal{L}(\mathfrak{d}_0|\mathfrak{d}|)/{\rm SL}_2(\mathbb{Z})} \chi_{\mathfrak{d}_0}(Q)\, r_{Q}(f_{2k}), 
\end{eqnarray*}
where in the second equation we have made use of the transformation law with respect to $z \mapsto p^{-1}z$. Therefore we obtain that ${\bf c}(\mathfrak{d}_0;\,{|\mathfrak{d}|}) :={\bf b}(\mathfrak{d}_0;\,{|\mathfrak{d}|}) \in \widetilde{R}^{\rm ord}$ satisfies the equation 
\begin{eqnarray*}
\lefteqn{
{\bf c}(\mathfrak{d}_0;\,{|\mathfrak{d}|})(\widetilde{P}) 
} \\
&=&{\Omega_{P} \over \Omega^{\epsilon}(P)}\,\left(1- \left({\,\mathfrak{d}\, \over p}\right) \beta_p(f_{2k})\, p^{-k}\right) \sum_{Q \in \mathcal{L}(\mathfrak{d}_0|\mathfrak{d}|)/{\rm SL}_2(\mathbb{Z})} \chi_{\mathfrak{d}_0}(Q)\, r_{Q}(f_{2k}) \\
&=&{\Omega_{P} \over \Omega^{\epsilon}(P)}\,\left(1- \left({\,\mathfrak{d}\, \over p}\right) \beta_p(f_{2k})\, p^{-k}\right) c_{|\mathfrak{d}|}(
\vartheta_{\mathfrak{d}_0}(f_{2k})).   
\end{eqnarray*}
(II) Suppose that $\mathfrak{d} \equiv 0 \pmod{p}$ and $\left({\mathfrak{d}_0|\mathfrak{d}|/p^2\over p}\right)=1$. Similarly to the case (I), in order to show that ${\bf c}(\mathfrak{d}_0;\,{|\mathfrak{d}|}):={\bf b}(\mathfrak{d}_0;\,{|\mathfrak{d}|})$ satisfies the desired property, it suffices to show the equation 
\begin{equation}
{\sum_{[a,b,c] \in \mathcal{L}_{p}(\mathfrak{d}_0|\mathfrak{d}|)/\Gamma_0(p)} \chi_{\mathfrak{d}_0}([a,b,c])\, 
\displaystyle\int_{C_{[a,b,c]}
} f_{2k}(pz)
(az^2+bz +c)^{k-1} dz}=0. 
\end{equation}
Indeed, we easily see that for each $s \in \mathbb{Z}/p\mathbb{Z}$, the cycle integrals on the left-hand side of the equation (12) are invariant under the translation $z \mapsto z+p^{-1}s$, and hence we have 
\begin{eqnarray*}
\lefteqn{
\sum_{[a,b,c] \in \mathcal{L}_{p}(\mathfrak{d}_0|\mathfrak{d}|)/\Gamma_0(p)} \chi_{\mathfrak{d}_0}([a,b,c])\, 
\displaystyle\int_{C_{[a,b,c]}
} f_{2k}(pz)
(az^2+bz +c)^{k-1} dz
} \\
&=& p^{-1} \sum_{[a,b,c] \in \mathcal{L}_{p}(\mathfrak{d}_0|\mathfrak{d}|)/\Gamma_0(p)} 
\left(\displaystyle\int_{C_{[a,b,c]}
} f_{2k}(pz)
(az^2+bz +c)^{k-1} dz \right)
\\
&& \times 
\sum_{s \in \mathbb{Z}/p\mathbb{Z}} 
\chi_{\mathfrak{d}_0}\!\left(\left[a,\,2(p^{-1}a) s + b,\,(p^{-2}a) s^2 + (p^{-1}b) s + c\right]\right). 
\end{eqnarray*}
Then it follows from (ii) of Lemma 3.6 and a simple calculation that   
\begin{eqnarray*}
\lefteqn{
\sum_{s \in \mathbb{Z}/p\mathbb{Z}} 
\chi_{\mathfrak{d}_0}\!\left(\left[a,\,2(p^{-1}a) s + b,\,(p^{-2}a) s^2 + (p^{-1}b) s + c\right]\right)
} \\ 
&=& \left({\mathfrak{d}_0/p \over a}\right) \sum_{s \in \mathbb{Z}/p\mathbb{Z}} \left({p \over (p^{-2}a) s + (p^{-1}b)s + c}\right) \\
&=& \left({\mathfrak{d}_0/p \over a}\right) 
\sum_{s \in \mathbb{Z}/p\mathbb{Z}} \left({p \over (p^{-1}b)s + c}\right)=0.
\end{eqnarray*}
Hence we obtain the desired equation (12). 
(III) Suppose that $\mathfrak{d} \equiv 0 \pmod{p}$ and $\left({\mathfrak{d}_0|\mathfrak{d}|/p^2\over p}\right)=-1$. We note that we cannot prove the assertion along the same lines as (I) and (II). 
However, we may fortunately take an alternative route as follows: 
We note that by virtue of Lemma 2.4 and the equation (11), there exists a fundamental discriminant $\mathfrak{d}_1$ with $(-1)^{k_0} \mathfrak{d}_1 >0$, 
 $\mathfrak{d}_1 \not\equiv 0 \pmod{p}$ and $c_{|\mathfrak{d}_1|}(\vartheta_{\mathfrak{d}_0}(f_{2{k_0}}))\ne 0$. 
Then by taking the $\mathfrak{d}$-th Shintani lifting 
 $\vartheta_{\mathfrak{d}}(f_{2k})$ as well as $\vartheta_{\mathfrak{d}_0}(f_{2k})$, the equation (9) yields 
\begin{eqnarray*}
c_{|\mathfrak{d}|}(\vartheta_{\mathfrak{d}_0}(f_{2k}))
&=& {
c_{|\mathfrak{d}_1|}(\vartheta_{\mathfrak{d}}(f_{2k})) \cdot c_{|\mathfrak{d}_0|}(\vartheta_{\mathfrak{d}_0}(f_{2k})) \over 
c_{|\mathfrak{d}_1|}(\vartheta_{\mathfrak{d}_0}(f_{2k}))} \\
&=&{
\left(1- \left({\,\mathfrak{d}_1\, \over p}\right) \beta_p(f_{2k})\, p^{-k}\right)
c_{|\mathfrak{d}_1|}(\vartheta_{\mathfrak{d}}(f_{2k})) \cdot c_{|\mathfrak{d}_0|}(\vartheta_{\mathfrak{d}_0}(f_{2k})) \over 
\left(1- \left({\,\mathfrak{d}_1\, \over p}\right) \beta_p(f_{2k})\, p^{-k}\right)
c_{|\mathfrak{d}_1|}(\vartheta_{\mathfrak{d}_0}(f_{2k}))},  
\end{eqnarray*}
where in the second equation of the above, we note that the $p$-adic interpolation properties of 
the numerator and the denominator 
have been already proved at the previous steps (I) and (II). 
Therefore,  
by taking a smaller analytic neighborhood $
\mathfrak{U}_0$ of $P_0 \in \mathfrak{X}_{\rm alg}(R^{\rm ord})$ if it's necessary to avoid possible  vanishing of the denominator, we obtain that 
the element 
\[
{\bf c}(\mathfrak{d}_0;\,{|\mathfrak{d}|}):={
{\bf b}(\mathfrak{d};\,{|\mathfrak{d}_1|}) \cdot 
{\bf b}(\mathfrak{d}_0;\,{|\mathfrak{d}_0|})
\over {\bf b}(\mathfrak{d}_0;\,{|\mathfrak{d}_1|})} \in (\widetilde{R}^{\rm ord})_{(\widetilde{P}_0)}
\]
defines a quotient of analytic functions on $\mathfrak{U}_0$ and satisfies the desired interpolation property. We complete the proof. 
\end{proof}

\section{Main results}

In this section, we construct a $\Lambda$-adic Duke-Imamo{\={g}}lu lifting for the universal ordinary $p$-stabilized newform ${\bf f}_{\rm ord} \in R^{\rm ord}[[q]]$. 
Therefore, we first introduce a suitable $p$-stabilization process for the Duke-Imamo{\={g}}lu lifting of ordinary 
Hecke eigenforms in $\mathscr{S}_{2k}({\rm SL}_2(\mathbb{Z}))$. 
Moreover, to consider the $p$-adic interpolation problem in the sequel, we give an explicit form of its Fourier expansion. 

\vspace*{3mm}
Suppose that positive integers $n$ and $k$ are fixed as $n \equiv k \pmod{2}$.  
Let $f
 \in \mathscr{S}_{2k}({\rm SL}_2(\mathbb{Z}))$ be a normalized Hecke eigenform ordinary at $p$, and $h
\in \mathscr{S}_{k+1/2}^{+}(\Gamma_0(4))^{(1)}$ a corresponding Hecke eigenform via the Shimura correspondence. As mentioned in \S\S 2.1, the Duke-Imamo{\={g}}lu lifting  
${\rm Lift}^{(2n)}(f) \in \mathscr{S}_{k+n}({\rm Sp}_{4n}(\mathbb{Z}))$ 
is characterized as a Hecke eigenform such that for each prime $l$, the Satake parameter $(\psi_0(l),\,\psi_1(l),\,\cdots,\,\psi_{2n}(l)) \in (\overline{\mathbb{Q}}^{\times})^{2n+1}/W_{2n}$ is written explicitly in terms of 
the ordered pair $(\alpha_{l}(f), \beta_{l}(f))$ 
satisfying 
$X^2 - a_l(f) X + l^{2k-1}=(X- \alpha_l(f))(X -\beta_l(f))$ and $v_l(\alpha_l(f)) \le v_l(\beta_l(f))$ (cf. Remark 2.4). In particular, if $l=p$, the ordinarity condition implies that $\alpha_p(f)$ is a $p$-adic unit, and hence 
$v_p(\beta_p(f))=2k-1$. Then we recall that 
the Hecke polynomial $\Phi_p(Y)\in \overline{\mathbb{Q}}_p^{\times}[Y]$ associated with ${\rm Lift}^{(2n)}(f)$ at $p$ is 
defined as 
\[
\Phi_p(Y):=(Y- \psi_0(p))\prod_{r=1}^{2n}\, \prod_{1 \le i_1 < \cdots < i_r \le 2n}
(Y-\psi_0(p)\psi_{i_1\!}(p) \cdots \psi_{i_r\!}(p)).  
\] 
We note that the equation 
(7) yields that $\psi_0(p)=p^{nk-n(n+1)/2}$ and the product
\begin{equation}
\psi_0(p) \prod_{i=1}^{n} \psi_i(p) 
=\alpha_p(f)^{n}
\end{equation}
is a unique unit $p$-adic root of $\Phi_p(Y)$. Put 
\begin{eqnarray*}
\Phi_p^{*}(Y)
&:=&
\Phi_p(Y) \cdot \{(Y-\psi_0(p))(Y - \psi_0(p) 
\prod_{i=1}^{n} \psi_i(p)
 )
\}^{-1} \\
&=&\prod_{r=1}^{2n}\, \prod_{\scriptstyle 1 \le i_1 < \cdots < i_r \le 2n, 
\atop {\scriptstyle (i_1,\cdots,i_n) \ne (1,\cdots,n)}}
(Y-\psi_0(p)\psi_{i_1\!}(p) \cdots \psi_{i_r\!}(p)), 
\end{eqnarray*}
and
\begin{eqnarray*}
\Psi_p^{*}(Y) 
&:=&
(Y- \psi_0(p) \left(\prod_{i=1}^{n} \psi_i(p)\right) \psi_{2n}(p))  \\
&& \times \prod_{j=1}^{n-1} (Y - \psi_0(p) \left(\prod_{i=1}^{n} \psi_i(p)\right) \psi_{2n-j}(p)\psi_{2n-j+1}(p)) \nonumber \\
&& \times
 (Y - \psi_0(p) \left(\prod_{i=1}^{n-1} \psi_i(p)\right) \psi_{2n}(p)). \nonumber 
\end{eqnarray*}
Obviously, we have $\Psi_p^{*}(Y)\, |\, \Phi_p^{*}(Y)\, |\, \Phi_p(Y)$ and it follows from the equation (7) that 
$\Phi_p^{*}(Y)$ and $\Psi_p^{*}(Y)$ can be taken of the forms 
in Theorem 1.3. 

\vspace*{3mm}
On the other hand, 
the following Hecke operators at $p$ are fundamental in the $p$-adic theory of Siegel modular forms of arbitrary genus $g \ge 1$:  
\[
U_{p,i}:=
\left\{\begin{array}{ll}
{I_{2g}\, 
{\rm diag}(\underbrace{1,\cdots,1}_{g},\underbrace{p,\cdots,p}_{g})
\, I_{2g}}, & \textrm{ if }i=0, \\[5mm]
{I_{2g}\, 
{\rm diag}(\underbrace{1,\cdots,1}_{i}\underbrace{p,\cdots,p}_{g-i},
\underbrace{p^2,\cdots,p^2}_{i}\underbrace{p,\cdots,p}_{g-i})
\, I_{2g}}, & 
\textrm{ if }1 \le i \le g-1, 
\end{array}\right.
\]
where $I_{2g}
:=\{M \in {\rm GSp}_{2g}(\mathbb{Z}_p)\, | \, M \textrm{ mod }p \in {\rm B}_{2g}(\mathbb{Z}/p\mathbb{Z})\}$. 
We note that these Hecke operators $U_{p,0}$, $U_{p,1},\,\cdots,\, U_{p,g-1}$ generate the Hecke algebra 
$$\mathscr{H}_p(I_{2g}, {\rm S}_{2g}):=\{\, I_{2g} M I_{2g}\, \,|\,M \in I_{2g} \backslash {\rm S}_{2g}
 / I_{2g}\,\}, $$
where ${\rm S}_{2g} \subset {\rm GSp}_{2g}(\mathbb{Q}_p)
$ is a semi-group such that  
\[
[I_{2g} \cap M^{-1}{\rm S}_{2g}M : I_{2g}],\, 
[I_{2g} \cap M^{-1}{\rm S}_{2g}M : M^{-1}{\rm S}_{2g}M] < +\infty
\] 
for all $M \in {\rm S}_{2g}$. 
In particular, we are interested in the operator $U_{p,0}$ which plays a central role in these operators. 
For instance, if $F=
\sum_{T \ge 0} A_{T}(F) q^T \in \mathscr{M}_{\kappa}(\Gamma_0(N))^{(g)}$ with $\kappa \ge 0$ and $N \ge 1$, 
the action of $U_{p,0}$ on $F$ admits the Fourier expansion
\begin{equation}
F|_{\kappa} U_{p,0}=\sum_{T \ge 0} A_{pT}(F) q^T, 
\end{equation}
which can be regarded as a generalization of the Atkin-Lehner $U_p$-operator acting on elliptic modular forms. 
Then we easily see that 
\[
F |_{\kappa} U_{p,0} \in \left\{
\begin{array}{ll}
\mathscr{M}_{\kappa}(\Gamma_0(N))^{(g)}, & \textrm{ if }p \,|\, N, \\[2mm]
\mathscr{M}_{\kappa}(\Gamma_0(Np))^{(g)}, & \textrm{ if }p \!\!\not|\, N,
\end{array}
\right.
\]
and the action of $U_{p,0}$ commutes with those of all Hecke operators at each prime $l \ne p$. 
For further details on the operator $U_{p,0}$
, see \cite{A-Z95,Boe05}. 

\vspace*{3mm}
Now, we define a $p$-stabilization of ${\rm Lift}^{(2n)}(f)$ by 
\begin{equation}
{\rm Lift}^{(2n)}(f)^{*}:={\Psi_p^{*}(\alpha_p(f)^n) \over \Phi_p^{*}(\alpha_p(f)^n)}\cdot {\rm Lift}^{(2n)}(f)|_{k+n}\,\Phi_p^{*}(U_{p,0}). 
\end{equation}
Obviously, the equation (13) yields that 
$$
{\Psi_p^{*}(\alpha_p(f)^n) \over 
\Phi_p^{*}(\alpha_p(f)^n)}
\ne 0. 
$$
In particular, if $n=1$, we have $\Psi_p^{*}(Y
) = \Phi_p^{*}(Y
)=(Y-p^k)(Y-\beta_p(f))$, and hence ${\Psi_p^{*}(\alpha_p(f)) /
\Phi_p^{*}(\alpha_p(f))}=1$. 
It also follows immediately from the above-mentioned properties of $U_{p,0}$ 
that 
${\rm Lift}^{(2n)}(f)^{*} \in \mathscr{S}_{k+n}(\Gamma_0(p))^{(2n)}$ is a Hecke eigenform such that 
for each prime $l \ne p$, the Hecke eigenvalues coincide with those of ${\rm Lift}^{(2n)}(f)$. 

\vspace*{3mm}
Then the Fourier expansion of ${\rm Lift}^{(2n)}(f)^{*}$ is written explicitly as follows: 

\begin{thm}
Under the same notation and assumptions as above, 
if $0<T \in {\rm Sym}_{2n}^{*}(\mathbb{Z})$ satisfies $\mathfrak{D}_T
=\mathfrak{d}_{T}\mathfrak{f}_{T}^2$, then we have 
\begin{eqnarray*}
A_{T}({\rm Lift}^{(2n)}(f)^{*})
&=& 
\left( 
1 - \left(
{\mathfrak{d}_T \over p}
\right)
\beta_p(f) p^{-k} \right) 
c_{|\mathfrak{d}_T|}(h)   \\
&& \times\, \alpha_p(f)^{v_p(\mathfrak{f}_T)+n(n+1)} \prod_{\scriptstyle \,\,l | \mathfrak{f}_{T}, \atop {\scriptstyle l \ne p}} 
\alpha_l(f)^{v_l(\mathfrak{f}_T)} F_l(T;\, l^{-k-n}\beta_l(f)). 
\end{eqnarray*}
\end{thm}

\begin{proof}
By virtue of the equations (6) and (13), 
we obtain 
\begin{eqnarray*}
\lefteqn{
A_{T}({\rm Lift}^{(2n)}(f)^{*})} \\
&=&
{\Psi_p^{*}(\alpha_p(f)^n) \over \Phi_p^{*}(\alpha_p(f)^n)}\,c_{|\mathfrak{d}_{T}|}(h) 
\prod_{\scriptstyle \,l\, | \mathfrak{f}_{T}, \atop {\scriptstyle l \ne p}} 
\alpha_l(f)^{v_l(\mathfrak{f}_T)}\, F_l(T;\,l^{-k-n}\beta_l(f)) 
\\
&& \times
\sum_{j=0}^{2^{2n}-2} (-1)^j 
s_j\!\left(\left\{\psi_0(p)\psi_{i_1}\!(p) \cdots \psi_{i_r}\!(p)\, \left|\, 
\begin{array}{c}
1\le r \le 2n, \\
1 \le i_1 < \cdots < i_r \le 2n, \\
(i_1,\cdots,i_n)\ne (1,\cdots,n)
\end{array}\right.
\right\}\right) \\
&& \hspace*{12mm}\times \alpha_p(f)^{v_p(\mathfrak{f}_T)+n(2^{2n}-2)-nj}
F_p(p^{2^{2n}-2-j}T;\,p^{-k-n} \beta_p(f)) \\
&=&
c_{|\mathfrak{d}_{T}|}(h) 
\alpha_p(f)^{v_p(\mathfrak{f}_T)+n(n+1)} 
\prod_{\scriptstyle \,l\, | \mathfrak{f}_{T}, \atop {\scriptstyle l \ne p}} 
\alpha_l(f)^{v_l(\mathfrak{f}_T)}\, F_l(T;\,l^{-k-n}\beta_l(f)) 
\\
&& \times 
{\alpha_p(f)^{-n(n+1)}\Psi_p^{*}(\alpha_p(f)^n) \over 
\alpha_p(f)^{-n(2^{2n}-2)}
\Phi_p^{*}(\alpha_p(f)^n)} \\
&& \times
\sum_{j=0}^{2^{2n}-2} (-1)^j 
s_j\!\left(\left\{\alpha_p(f)^{-n}\psi_0(p)\psi_{i_1}\!(p) \cdots \psi_{i_r}\!(p)\, \left|\, 
\begin{array}{c}
1\le r \le 2n, \\
1 \le i_1 < \cdots < i_r \le 2n, \\
(i_1,\cdots,i_n)\ne (1,\cdots,n)
\end{array}\right.
\right\}\right) \\
&& \hspace*{12mm}\times 
F_p(p^{2^{2n}-2-j}T;\,p^{-k-n} \beta_p(f)), 
\end{eqnarray*}
where $s_j(\{X_1,\cdots,X_{2^{2n}-2}\})$ denotes the $j$-th elementary symmetric polynomial in variables $X_1,\,\cdots,\,X_{2^{2n}-2}$. Here by the equation (7), we easily see that 
$\alpha_p(f)^{-n}\psi_0(p)=p^{n(n+1)/2}\{\beta_p(f) p^{-k-n}\}^n$,  
\[
\psi_i(p)=\left\{\begin{array}{ll}
p^{-i}\{\beta_p(f) p^{-k-n}\}^{-1}, & \textrm{if }1 \le i \le n, \\[2mm]
p^{i}\{\beta_p(f) p^{-k-n}\}, & \textrm{if }n+1 \le i \le 2n,   
\end{array}\right.
\]
and \begin{eqnarray*}
\lefteqn{
\alpha_p(f)^{-n(n+1)}\Psi_p^{*}(\alpha_p(f)^n)
} \\
&=&
(1-p^{2n} \{\beta_p(f) p^{-k-n}\})\prod_{i=1}^{n}(1- p^{2n+2i-1}\{\beta_p(f) p^{-k-n}\}^2). 
\end{eqnarray*}
Put 
\[
\xi_{p,i}(X):=\left\{\begin{array}{ll}
p^{n(n+1)/2}X^n,  & \textrm{if }i =0, \\[2mm]
p^{-i}X^{-1}, & \textrm{if }1 \le i \le n, \\[2mm]
p^{i}X, & \textrm{if }n+1 \le i \le 2n,   
\end{array}\right.
\]
Hence it suffices to show that the equation  
\begin{eqnarray*}
\lefteqn{(1-p^{2n}X)\prod_{i=1}^{n}(1- p^{2n+2i-1}X^2) } \\
&\times& \hspace*{-2mm}\sum_{j=0}^{2^{2n}-2} (-1)^j  
s_j
\!\!\left(\left\{\xi_{p,0}(X)\xi_{p,i_1\!}(X) \cdots \xi_{p,i_r\!}(X)\, \left|\, 
\begin{array}{c}
1\le r \le 2n, \\
1 \le i_1 < \cdots < i_r \le 2n, \\
(i_1,\cdots,i_n)\ne (1,\cdots,n)
\end{array}\right.
\right\}\right) 
\nonumber \\[3mm]
&& \hspace*{10mm}
\times F_p(p^{2^{2n}-2-j}T;\,X) \\[3mm]
&=&\left(1 - \left(
{\mathfrak{d}_{T} \over p}
\right)
 p^n X\right)\prod_{r=1}^{2n}\, \prod_{\scriptstyle 1 \le i_1 < \cdots < i_r \le 2n, 
\atop {\scriptstyle (i_1,\cdots,i_n) \ne (1,\cdots,n)}}
{(1-\xi_{p,0}(X)\xi_{p,i_1\!}(X) \cdots \xi_{p,i_r\!}(X))} \nonumber
\end{eqnarray*}
holds for each $T \in {\rm Sym}_{2n}^{*}(\mathbb{Z})$ with $T>0$. 
First, we prove the assertion for $n=1$. 
In that case, 
it suffices to show the equation 
\[
F_p(p^2 T;\, X)
-(p^2 X+p^3 X^2) F_p(p T;\,X) 
+p^5 X^3 F_p(T;\,X)
=1 - \left(
{\mathfrak{d}_{T} \over p}
\right)
 p X.
 \]
Indeed, for each $T \in {\rm Sym}_2^{*}(\mathbb{Z})$ with $T >0$, the polynomial $F_p(T;\,X)$ admits the explicit form 
\begin{eqnarray*}
\lefteqn{
F_p(T;\,X) } \\
&=& \sum_{i=0}^{v_p(\mathfrak{m}_T)} (p^2 X)^i \left\{ \sum_{j=0}^{v_p(\mathfrak{f}_T) - i} (p^3 X^2)^j
-\left({\mathfrak{d}_T \over p}\right)\,pX 
\sum_{j=0}^{v_p(\mathfrak{f}_T) - i-1} (p^3 X^2)^j \right\}, 
\end{eqnarray*}
where $\mathfrak{m}_T = \max\{ 0<m \in \mathbb{Z} \,|\, m^{-1} T \in {\rm Sym}_2^{*}(\mathbb{Z})\}$ (cf. \cite{Kau59
}). Therefore we have 
\begin{eqnarray*}
\lefteqn{
F_p(p^2 T;\, X)
-(p^2 X+p^3 X^2) F_p(p T;\,X) 
+p^5 X^3 F_p(T;\,X) } \\
&=& \sum_{i=0}^{v_p(\mathfrak{m}_T)+2} (p^2 X)^i 
\left\{ \sum_{j=0}^{v_p(\mathfrak{f}_T) - i+2} (p^3 X^2)^j
-\left(
{\mathfrak{d}_{T} \over p}
\right)\,pX 
\sum_{j=0}^{v_p(\mathfrak{f}_T) - i+1} (p^3 X^2)^j \right\} \\
&& \hspace*{5mm} -\sum_{i=0}^{v_p(\mathfrak{m}_T)+1} (p^2 X)^{i+1} 
\left\{ \sum_{j=0}^{v_p(\mathfrak{f}_T) - i+1} (p^3 X^2)^j
-\left(
{\mathfrak{d}_{T} \over p}
\right)\,pX 
\sum_{j=0}^{v_p(\mathfrak{f}_T) - i} (p^3 X^2)^j \right\} \\
&& - \sum_{i=0}^{v_p(\mathfrak{m}_T)+1} (p^2 X)^i 
\left\{ \sum_{j=0}^{v_p(\mathfrak{f}_T) - i+1} (p^3 X^2)^{j+1}
-\left(
{\mathfrak{d}_{T} \over p}
\right)\,pX 
\sum_{j=0}^{v_p(\mathfrak{f}_T) - i} (p^3 X^2)^{j+1} \right\} \\
&& \hspace*{5mm} + \sum_{i=0}^{v_p(\mathfrak{m}_T)} (p^2 X)^{i+1} 
\left\{ \sum_{j=0}^{v_p(\mathfrak{f}_T) - i} (p^3 X^2)^{j+1}
-\left(
{\mathfrak{d}_{T} \over p}
\right)\,pX 
\sum_{j=0}^{v_p(\mathfrak{f}_T) - i-1} (p^3 X^2)^{j+1} \right\} \\
&=& \sum_{i=0}^{v_p(\mathfrak{m}_T)+2} (p^2 X)^i 
\left\{ \sum_{j=0}^{v_p(\mathfrak{f}_T) - i+2} (p^3 X^2)^j
-\left(
{\mathfrak{d}_{T} \over p}
\right)\,pX 
\sum_{j=0}^{v_p(\mathfrak{f}_T) - i+1} (p^3 X^2)^j \right\} \\
&& \hspace*{5mm} -\sum_{i=1}^{v_p(\mathfrak{m}_T)+2} (p^2 X)^{i} 
\left\{ \sum_{j=0}^{v_p(\mathfrak{f}_T) - i+2} (p^3 X^2)^j
-\left(
{\mathfrak{d}_{T} \over p}
\right)\,pX 
\sum_{j=0}^{v_p(\mathfrak{f}_T) - i+1} (p^3 X^2)^j \right\} \\
&& - \sum_{i=0}^{v_p(\mathfrak{m}_T)+1} (p^2 X)^i 
\left\{ \sum_{j=0}^{v_p(\mathfrak{f}_T) - i+1} (p^3 X^2)^{j+1}
-\left(
{\mathfrak{d}_{T} \over p}
\right)\,pX 
\sum_{j=0}^{v_p(\mathfrak{f}_T) - i} (p^3 X^2)^{j+1} \right\} \\
&& \hspace*{5mm} + \sum_{i=1}^{v_p(\mathfrak{m}_T)+1} (p^2 X)^i 
\left\{ \sum_{j=0}^{v_p(\mathfrak{f}_T) - i+1} (p^3 X^2)^{j+1}
-\left(
{\mathfrak{d}_{T} \over p}
\right)\,pX 
\sum_{j=0}^{v_p(\mathfrak{f}_T) - i} (p^3 X^2)^{j+1} \right\} \\ %
&=& 
\left\{ \sum_{j=0}^{v_p(\mathfrak{f}_T)+2} (p^3 X^2)^j
-\left(
{\mathfrak{d}_{T} \over p}
\right)\,pX 
\sum_{j=0}^{v_p(\mathfrak{f}_T)+1} (p^3 X^2)^j \right\} \\
&& \hspace*{5mm} - \left\{ \sum_{j=0}^{v_p(\mathfrak{f}_T)+1} (p^3 X^2)^{j+1}
-\left(
{\mathfrak{d}_{T} \over p}
\right)\,pX 
\sum_{j=0}^{v_p(\mathfrak{f}_T)} (p^3 X^2)^{j+1} \right\} \\
&=& 
\left\{ \sum_{j=0}^{v_p(\mathfrak{f}_T)+2} (p^3 X^2)^j
-\left(
{\mathfrak{d}_{T} \over p}
\right)\,pX 
\sum_{j=0}^{v_p(\mathfrak{f}_T)+1} (p^3 X^2)^j \right\} \\
&& \hspace*{5mm} - \left\{ \sum_{j=1}^{v_p(\mathfrak{f}_T)+2} (p^3 X^2)^j
-\left(
{\mathfrak{d}_{T} \over p}
\right)\,pX 
\sum_{j=1}^{v_p(\mathfrak{f}_T)+1} (p^3 X^2)^j \right\} \\
&=& 1 - \left(
{\mathfrak{d}_{T} \over p}
\right) p X.  
\end{eqnarray*}
For each $n >1$, it follows from the equation (4) that the desired equation is equivalent to the following: 
\begin{eqnarray*}
\lefteqn{\sum_{j=0}^{2^{2n}-2} (-1)^j  
s_j\!
\left(\left\{\xi_{p,0}(X)\xi_{p,i_1\!}(X) \cdots \xi_{p,i_r\!}(X)\, \left|\, 
\begin{array}{c}
1\le r \le 2n, \\
1 \le i_1 < \cdots < i_r \le 2n, \\
(i_1,\cdots,i_n)\ne (1,\cdots,n)
\end{array}\right.
\right\}\right)} \\[2mm]
&& \hspace*{2mm}
\times b_p(p^{2^{2n}-2-j}T;\,X) \\[2mm]
&=& { 
\displaystyle(1-X)\prod_{i=1}^{n}(1- p^{2i}X^2) 
\prod_{r=1}^{2n}\, \prod_{\scriptstyle 1 \le i_1 < \cdots < i_r \le 2n, 
\atop {\scriptstyle (i_1,\cdots,i_n) \ne (1,\cdots,n)}}
(1-\xi_{p,0}(X)\xi_{p,i_1\!}(X) \cdots \xi_{p,i_r\!}(X))
\over 
(1-p^{2n}X)\displaystyle\prod_{i=1}^{n}(1- p^{2n+2i-1}X^2)
}. \nonumber
\end{eqnarray*}
This can be proved by making use of the same arguments as in \cite{Zha75} and 
\cite{Kit86} (see also \cite{B-S87}).    
Now we complete the proof. 
\end{proof}

\vspace*{2mm}
As a consequence of Theorem 4.1, we also have  

\begin{cor}Under the same assumption as above, 
we have
\[
{\rm Lift}^{(2n)}(f)^{*}|_{k+n}U_{p,0}
 = 
 \alpha_p(f)^n 
\cdot {\rm Lift}^{(2n)}(f)^{*}. 
\]
\end{cor}

\vspace*{2mm}
Indeed, the desired equation follows immediately from Theorem 4.1. \hfill $\Box$

\vspace*{3mm}
We should mention that Courtieu-Panchishkin \cite{C-P04} stated a general philosophy 
of the $p$-stabilization for Siegel modular forms. In accordance with it, we may also consider another $p$-stabilization 
\begin{eqnarray*}
{\rm Lift}^{(2n)}(f)^{\dag}&:=& {\rm Lift}^{(2n)}(f)|_{k+n} (Y-\psi_0(p))\cdot\Phi_p^{*}(U_{p,0}) \\
&=& {\rm Lift}^{(2n)}(f)|_{k+n} \Phi_p(U_{p,0})\cdot(U_{p,0}-\psi_0(p)\prod_{i=1}^n \psi_i(p))^{-1}. 
\end{eqnarray*}
Since $U_{p,0}$ annihilates the Hecke polynomial $\Phi_p(X)$ (cf. Proposition 6.10 in \cite{A-Z95}), 
it follows from the equation (13) that 
\begin{eqnarray*}
{\rm Lift}^{(2n)}(f)^{\dag}|_{k+n}U_{p,0}
 &=& \psi_0(p) \prod_{i=1}^{n} \psi_i(p)
\cdot {\rm Lift}^{(2n)}(f)^{\dag} \\
&=&\alpha_p(f)^n 
\cdot {\rm Lift}^{(2n)}(f)^{\dag}.  
\end{eqnarray*}
We easily verify that $U_{p,0}$ annihilates $\Phi_p^*(X)\cdot (X-\psi_0(p)\prod_{i=1}^n \psi_{i}(p)) 
$ as well, and hence we may also prove Corollary 4.2 along the same line as above. 
Since we have 
\[
{\Psi_p^{*}(\alpha_p(f)^n) \over \Phi_p^{*}(\alpha_p(f)^n)}\cdot{\rm Lift}^{(2n)}(f)^{\dag} = (1- p^{nk-n(n+1)/2})\cdot 
{\rm Lift}^{(2n)}(f)^{*}, 
\]
${\rm Lift}^{(2n)}(f)^{*}$ and ${\rm Lift}^{(2n)}(f)^{\dag}$ are essentially the same, however, the former can be regarded as the 
principle $p$-stabilization for ${\rm Lift}^{(2n)}(f)$ and the Siegel Eisenstein series $E_{k+n}^{(2n)}$ (cf. \S\S5.1 below).  

\begin{rem}[semi-ordinarity]
Unfortunately, even if $f$ is ordinary (and so is the associated Galois representation), ${\rm Lift}^{(2n)}(f)$ does not admit the ordinary $p$-stabilization in the sense of Hida (cf. \cite{Hid02,Hid04}). 
That is, ${\rm Lift}^{(2n)}(f)$ may not be an eingenfunction of the Hecke operators $U_{p,0},\,U_{p,1}\cdots,\,U_{p,2n-1}$ 
such that all the eigenvalues are $p$-adic units. It is because the fact that the $p$-local spherical representation associated with ${\rm Lift}^{(2n)}(f)$ is not equal to any induced representation of an unramified character.  
On the other hand, Corollary 4.2 implies that 
there 
exists a $p$-stabilization 
${\rm Lift}^{(2n)}(f)^{*}$ so that the eigenvalue of $U_{p,0}$ is a $p$-adic unit. 
Following \cite{S-U06}, we refer to it as the {\it semi-ordinary} $p$-stabilization of ${\rm Lift}^{(2n)}(f)$. 
In fact, it turns out that this condition is sufficient to adapt the ordinary theory. Since the operator $U_{p,0}$ is given by the trace of the Frobenius operator acting on the ordinary part of the cohomology of the Siegel variety, there is no need to use the overconvergence of canonical subgroup. 
\end{rem}


Now, for the universal ordinary $p$-stabilized newform 
${\bf f}_{\rm ord}=\sum_{m \ge 1}{\bf a}_m q^m \in R^{\rm ord}[[q]]$ interpolating the Hida family $\{f_{2k}^{*}\}$, 
we construct its $\Lambda$-adic lifting interpolating 
the semi-ordinary $p$-stabilized Duke-Imamo{\={g}}lu lifting 
$\{{\rm Lift}(f_{2k})^{*}\}$: 

\begin{thm} 
For each integer $n \ge 1$, and a fixed integer ${k_0 > n+1}$ with $k_0 \equiv n \pmod{2}$, let
$\mathfrak{d}_0$ and $\mathfrak{U}_0$ be a pair of a fundamental discriminant and an analytic neighborhood of 
$P_0 \in \mathfrak{X}_{\rm alg}(R^{\rm ord})$ taken as in Proposition 3.7. For each integer $k \ge k_0$ with 
$k \equiv k_0 \pmod{2}$, we denote by ${\rm Lift}_{\mathfrak{d}_0}^{(2n)}(f_{2k})$ the Duke-Imamo{\={g}}lu lifting of $f_{2k}$ associated with the ${\mathfrak{d}_0}$-th Shintani lifting 
$\vartheta_{\mathfrak{d}_0}(f_{2k})$. 
Then for $\widetilde{P}_0=(k_0, \omega^{k_0})\in \mathfrak{X}_{\rm alg}(\widetilde{R}^{\rm ord})$ lying over $P_0$, there exist a formal Fourier expansion $${\bf F}=\sum_{T >0} {\bf a}_T\, q^T \in (\widetilde{R}^{\rm ord})_{(\widetilde{P}_0)}[[q]]^{(2n)}$$ and a choice of $p$-adic period $\Omega_{P} \in \overline{\mathbb{Q}}_p$ for $P \in \mathfrak{U}_0
$ satisfying the following: 
\begin{itemize}
\item[(i)] $\Omega_{P_0} \ne 0$.   
 
\item[(ii)] For each $\widetilde{P}=(k, \omega^{k}) \in \mathfrak{X}_{\rm alg}(\widetilde{R}^{\rm ord})$ lying over $P=(2k, \omega^{2k}) \in \mathfrak{U}_0$ with 
$k \equiv k_0 \pmod{2}$, we have
\[
{\bf a}_T(\widetilde{P})={\Omega_{P} \over \Omega^{\epsilon}(P)}\,
A_T({\rm Lift}_{\mathfrak{d}_0}^{(2n)}(f_{2k})^{*}),
\]
where $\Omega^{\epsilon}(P) \in \mathbb{C}^{\times}$ is the complex period of $P$ with signature $\epsilon \in \{\pm \}$. 
\end{itemize}
\end{thm}

\begin{proof}
By combining Theorem 4.1 with the equation (9), we have 
\begin{eqnarray*}
\lefteqn{
A_{T}({\rm Lift}^{(2n)}(f_{2k})^{*})
} \\
&=& 
\left( 
1 - \left(
{\mathfrak{d}_T \over p}
\right)
\beta_p(f_{2k}) p^{-k} \right) 
c_{|\mathfrak{d}_T|}(\vartheta_{\mathfrak{d}_0}(f_{2k}))   \\
&& 
\times\, \alpha_p(f_{2k})^{v_p(\mathfrak{f}_T)+n(n+1)} \prod_{\scriptstyle l | \mathfrak{f}_{T}, \atop {\scriptstyle l \ne p}} 
\alpha_l(f_{2k})^{v_l(\mathfrak{f}_T)} F_l(T, l^{-k-n}\beta_l(f_{2k})) \\
&=& 
\left( 
1 - \left(
{\mathfrak{d}_T \over p}
\right)
\beta_p(f_{2k}) p^{-k} \right) 
c_{|\mathfrak{d}_T|}(\vartheta_{\mathfrak{d}_0}(f_{2k}))   \\
&& 
\times\, \alpha_p(f_{2k})^{n(n+1)} \prod_{\scriptstyle l | \mathfrak{f}_{T}, \atop {\scriptstyle l \ne p}} 
\sum_{i=0}^{v_l(\mathfrak{f}_T)} \phi_T(l^{v_l(\mathfrak{f}_T)-i}) 
(l^{v_l(\mathfrak{f}_T)-i})^{k-1} a_{l^{i}}(f).  
\end{eqnarray*}
We easily see that for each prime $l \ne p$ and each integer $r\ge 0$, there exists an element ${\bf d}({l^r}) \in \Lambda
$ such that 
${\bf d}(l^r)(P) = l^{r(k-1)}$ for each $P=(2k ,\, \omega^{2k}) \in \mathfrak{X}_{\rm alg}(R^{\rm ord})$.  
We define an element ${\bf a}_{T}\in (\widetilde{R}^{\rm ord})_{(\widetilde{P}_0)}$ by 
\begin{equation}
{\bf a}_{T}:={\bf c}(\mathfrak{d}_0;\,|\mathfrak{d}|)\, {\bf a}_p^{n(n+1)} \prod_{\scriptstyle l | \mathfrak{f}_{T}, \atop {\scriptstyle l \ne p}} 
\sum_{i=0}^{v_l(\mathfrak{f}_T)} \phi_T(l^{v_l(\mathfrak{f}_T)-i}) 
{\bf d}(l^{v_l(\mathfrak{f}_T)-i})\, {\bf a}_{l^{i}},  
\end{equation}
where ${\bf c}(\mathfrak{d}_0;\,|\mathfrak{d}|) \in (\widetilde{R}^{\rm ord})_{(\widetilde{P}_0)}$ is the element in Proposition 3.7. Then the desired interpolation property follows immediately from Theorem II and Proposition 3.7, and we complete the proof. 
\end{proof}

According to the explicit form (16), the $\Lambda$-adic Duke-Imamo{\={g}}lu lifting admits a non-vanishing property as well as 
the $\Lambda$-adic Shintani lifting that we have already established in \S\S 3.2. Therefore we consequently obtain 
a semi-ordinary $p$-adic analytic family of non-zero cuspidal Hecke eigenforms 
$
\{{\Omega_{P} \over \Omega^{\epsilon}(P)}\,
{\rm Lift}_{\mathfrak{d}_0}^{(2n)}(f_{2k})^{*} \in \mathscr{S}_{k+n}(\Gamma_0(p))^{(2n)}\}
$
parametrized by varying $P=(2k,\,\omega^{2k}) \in \mathfrak{U}_0$ with 
$k \equiv k_0 \equiv n \pmod{2}$. 

\begin{rem}
Recently, Ikeda \cite{Ike10} generalized the classical Duke-Imamo{\={g}}lu lifting to 
a Langlands functorial lifting of cuspidal automorphic representations of ${\rm PGL}_2(\mathbb{A}_K)$ to ${\rm Sp}_{4n}(\mathbb{A}_K)$ over a totally real algebraic field $K$, which contains some 
generalizations of the classical Duke-Imamo{\={g}}lu lifting for $\mathscr{S}_{2k}(\Gamma_0(N))^{(1)}$ with $N \ge 1$. 
Unfortunately, the case $N=p$ has been excluded from the framework so far.  
It is because 
resulting automorphic representations of ${\rm Sp}_{4n}(\mathbb{A}_K)$ are likely to have 
the Steinberg representations as their $p$-local components. 
Therefore we cannot use the 
Duke-Imamo{\={g}}lu lifting of 
$f_{2k}^{*} \in \mathscr{S}_{2k}(\Gamma_0(p))^{(1)}$ directly. However, by virtue of Theorem 4.1 and Corollary 4.2, 
we may regard the semi-ordinary 
$p$-stabilized Duke-Imamo{\={g}}lu lifting 
${\rm Lift}_{\mathfrak{d}_0}^{(2n)}(f_{2k})^{*}\in \mathscr{S}_{k+n}(\Gamma_0(p))^{(2n)}$ naturally as 
a possible generalized Duke-Imamo{\={g}}lu lifting of $
f_{2k}^{*} \in \mathscr{S}_{2k}^{\rm old}(\Gamma_0(p))^{(1)}$. The author expects that 
the theory of the $p$-adic stabilization for Siegel modular forms is not only interesting in its own right 
but also useful in the study of classical Siegel modular forms and the associated automorphic representations.  
\end{rem}

\section{
Appendix
}

\subsection{$\Lambda$-adic Siegel Eisenstein series of even genus}

As applications of the methods we have used in the previous \S 4, we give a similar result for the Siegel Eisenstein series of even genus. 

\vspace*{2mm}Recall that for the 
Eisenstein series $E_{2k}^{(1)}  \in \mathscr{M}_{2k}({\rm SL}_2(\mathbb{Z}))$, the ordinary $p$-stabilization  
\[
(E_{2k}^{(1)})^{*}(z):=E_{2k}^{(1)}(z)-p^{2k-1}E_{2k}^{(1)}(pz) \in \mathscr{M}_{2k}(\Gamma_0(p))^{(1)} \hspace*{3mm} (z \in \mathfrak{H}_1)
\]
can be assembled into a $p$-adic analytic family (cf. \cite{Hid93}). On the other hand, 
we have mentioned in \S\S 2.1 that for each pair of positive integers $n,\,k$ with $k > n+1$ and $k \equiv n \pmod{2}$, the Siegel Eisenstein series $E_{k+n}^{(2n)}\in \mathscr{M}_{k+n}({\rm Sp}_{4n}(\mathbb{Z}))$ can be formally regarded as 
the Duke-Imamo{\={g}}lu lifting of $E_{2k}^{(1)}$. 
Therefore by replacing the Hida family $\{f_{2k}^{*}\}$ with $\{(E_{2k}^{(1)})^{*}\}$, we may naturally deduce the following: 

\begin{thm} 
For each positive integers $n$ and $k$ with $k>n+1$ and $k \equiv n \pmod{2}$, put
\[
(E_{k+n}^{(2n)})^{*}:=
{\Psi_p^{*}(1) \over \Phi_p^{*}(1)}\cdot E_{k+n}^{(2n)}|_{k+n} \Phi_p^{*}(U_{p,0}), 
\]
where $\Phi_p^{*}(Y)$, $\Psi_p^{*}(Y) \in \overline{\mathbb{Q}}_p^{\times}[Y]$ denote the polynomials defined in \S 4, but for $(\alpha_p(E_{2k}^{(1)}),\beta_p(E_{2k}^{(1)}))=(1,p^{2k-1})$. 
Then we have
\begin{enumerate}[\upshape (i)]
\item 
$(E_{k+n}^{(2n)})^{*}$ is a non-cuspidal Hecke eigenform in $\mathscr{M}_{k+n}(\Gamma_0(p))^{(2n)}$ such that the Hecke eigenvalues agree with $E_{k+n}^{(2n)}$ for each $l \ne p$ and 
\[
(E_{k+n}^{(2n)})^{*}|_{k+n}U_{p,0}=(E_{k+n}^{(2n)})^{*}. 
\]
\item Let $\mathcal{L}={\rm Frac}(\Lambda)$ be the field of fractions of $\Lambda$. 
If $0 \le T \in {\rm Sym}_{2n}^{*}(\mathbb{Z})$ satisfies either ${\rm rank}(T)=0$ or ${\rm rank}(T)=2n$ ${\rm (i.e.\, }T>0{\rm )}$, 
there exists an element ${\bf e}_T \in \mathcal{L}$ such that 
\[
{\bf e}_T(P_{2k})=A_T((E_{k+n}^{(2n)})^{*})
\]
for each $P_{2k} \in \mathfrak{X}_{\rm alg}(\Lambda)$ with $P_{2k}(\gamma)=\gamma^{2k}$.  
\end{enumerate}
\end{thm}

\begin{proof}By making use of the fundamental properties of $U_{p,0}$ 
stated as in \S 4, we easily obtain the assertion (i). 
For each $T \in {\rm Sym}_{2n}^{*}(\mathbb{Z})$ with ${\rm rank}(T)=2n$, by 
applying Theorem 4.1 for $E_{2k}^{(1)}$ instead of $f_{2k}$, we have
\begin{eqnarray*}
A_{T}((E_{k+n}^{(2n)})^{*})
&=& 
\left( 
1 - \left(
{\mathfrak{d}_T \over p}
\right)
p^{k-1}
 \right) 
L(1-k,\left(
{\mathfrak{d}_T \over *}
\right))
 \prod_{\scriptstyle \,\,l | \mathfrak{f}_{T}, \atop {\scriptstyle l \ne p}} 
F_l(T;\, 
l^{k-n-1}
) \\
&=& 
L^{(p)}(1-k,\left(
{\mathfrak{d}_T \over *}
\right))
 \prod_{\scriptstyle \,\,l | \mathfrak{f}_{T}, \atop {\scriptstyle l \ne p}} 
F_l(T;\, l^{k-n-1}
),
\end{eqnarray*}
where $L^{(p)}(s,\left(
{\mathfrak{d}_T \over *}
\right)):=L(s,\left(
{\mathfrak{d}_T \over *}
\right))(1- \left(
{\mathfrak{d}_T \over p}
\right) p^{-s})$.
If $T \in {\rm Sym}_{2n}^{*}(\mathbb{Z})$ has ${\rm rank}(T)=0$, then it follows from the definition of $E_{k+n}^{(2n)}$ that 
\begin{eqnarray*}
A_T((E_{k+n}^{(2n)})^{*}) 
&=& {\Psi_p^{*}(1) \over \Phi_p^{*}(1)} \cdot \Phi_p^{*}(1) \cdot A_T(E_{k+n}^{(2n)}) \\
&=& 2^{-n} \zeta^{(p)}(1-k-n) \prod_{i=1}^{n} \zeta^{(p)}(1-2k-2n+2i), 
\end{eqnarray*}
where $\zeta^{(p)}(s):=\zeta(s) (1-p^{-s})$. 
Therefore, the
assertion (ii) follows immediately from the constructions of $p$-adic $L$-functions in the sense of Kubota-Leopoldt. 
\end{proof}

\begin{rem} More generally in the same context, Panchishkin \cite{Pan00} has already obtained a similar result for the twisted Siegel Eisenstein series of arbitrary genus $g$ by a cyclotomic character. However, in that case, we may conduce only the $p$-adic interpolation properties of the Fourier coefficients for each positive definite $T \in {\rm Sym}_{g}^{*}(\mathbb{Z})$ with $p \!\not\!|\, \det(T)$. 
Fortunately, in our setting, the description of the $p$-stabilization process allows us to resolve the $p$-adic interpolation problem for more general Fourier coefficients, but only for $T \in {\rm Sym}_{2n}^{*}(\mathbb{Z})$ with rank occupying in extreme cases. 
\end{rem}

When $n=1$, by virtue of \cite{Kau59}, we obtain a complete 
satisfactory result for the $\Lambda$-adic Siegel Eisenstein series of genus $2$ and of level $1$ as follows: 

\begin{cor}
There exists a formal Fourier expansion  
\[{\bf E}=\sum_{T \ge 0} 
{\bf e}
_T
\, q^T 
\in \mathcal{L}[[q]]^{(2)}
\] 
such that  
\[
{\bf E}(P_{2k})=(E_{k+1}^{(2)})^{*}
\] 
for each $P_{2k} \in \mathfrak{X}_{\rm alg}(\Lambda)$ with $k > 2$ and $k \equiv 1 \pmod{2}$. 
\end{cor}

Indeed, by virtue of Theorem 5.1, it suffices to show the assertion 
for each $0 \le T \in {\rm Sym}_{2}^{*}(\mathbb{Z})$
with ${\rm rank}(T)=1$. Then Theorem in \cite{Kau59} yields 
that $A_T((E_{k+1}^{(2)})^{*})$ can be written as 
\[
A_T(E_{k+1}^{(2)})=
\zeta(1-2k) \displaystyle\prod_{\scriptstyle l | \mathfrak{m}_T
} F_l^{(1)}(\mathfrak{m}_T;\, l^{k-1}), 
\] 
where $\mathfrak{m}_T = \max\{ 0<m \in \mathbb{Z} \,|\, m^{-1} T \in {\rm Sym}_2^{*}(\mathbb{Z})\}$ and  
$F_l^{(1)}(b,\, X)=\sum_{i=0}^{v_l(b)} (lX)^i$ for each integer $b>0$. 
Since we easily see that 
\[
F_p^{(1)}(p^2 \mathfrak{m}_T;\, pX) - (p^2 X + p^3 X^2) F_p^{(1)}(p \mathfrak{m}_T;\,pX)+p^5 X^3 F_p^{(1)}(\mathfrak{m}_T;\,pX)=1 -p^3 X^2,
\]
we have 
\[
A_T((E_{k+1}^{(2)})^{*})=
\zeta^{(p)}(1-2k) \displaystyle\prod_{\scriptstyle l | \mathfrak{m}_T, \atop {\scriptstyle l \ne p}} F_l^{(1)}(\mathfrak{m}_T;\, l^{k-1}).  
\] 
Hence we obtain the assertion. \hfill $\Box$

\vspace*{3mm}
For each $n >1$, the author does not 
know the explicit form of 
the Fourier coefficient $A_T(E_{k+n}^{(2n)})$ for $0 \le T \in {\rm Sym}_{2n}^{*}(\mathbb{Z})$ with $1 \le  {\rm rank}(T) <2n$. However, the $\Lambda$-adic Siegel Eisenstein series which interpolates the whole Fourier expansion of 
$(E_{k+n}^{(2n)})^{*}$ is expected to exist in general. 

\begin{rem}
Along the same line as above, Takemori \cite{Tak10} independently proved that 
for each integer $r \ge 0$ and each integer $N \ge 1$ prime to $p$, 
a similar result also holds for the Siegel Eisenstein series of genus $2$ and of level $Np^r$ with a primitive character. In the present article, we mainly dealt with the 
Duke-Imamo{\={g}}lu lifting according to \cite{Ike01}, which requires the conditions $r=0$ and $N=1$. Therefore the Siegel Eisenstein series has been discussed under the same conditions, but for arbitrary even genus $2n \ge2$. We note that the method we use is also extendable to the Siegel Eisenstein series of even genus in a more general setting, at least, for $N>1$. 
\end{rem}


\subsection{Numerical evidences}
Finally in this section, we present some numerical evidences of the $p$-adic analytic family $\{{\rm Lift}^{(2n)}(f_{2k})^{*}\}$ which have been done by using the Wolfram {\it Mathematica 7}. 

\begin{ex}
As mentioned in Example 3.4, 
Ramanujan's $\Delta$-function $f_{12} \in \mathscr{S}_{12}({\rm SL}_2(\mathbb{Z}))$ and 
a normalized Hecke eigenform $f_{32} \in \mathscr{S}_{32}({\rm SL}_2(\mathbb{Z}))$ can be assembled into the same Hida family for $p=11$. 
We easily see that 
\begin{eqnarray*}
h_{13/2}&=&q - 56 q^4 + 120 q^5 - 240 q^8 +9 q^9 + 1440 q^{12}+\cdots \in \mathbb{Z}[[q]],\\
h_{33/2}&=&q + (x-32768) q^4 + (2x-65568) q^5 + (218x-7116672) q^8  \\
&& +(432x-14298687) q^9 -(2916x-103037184) q^{12}+\cdots \in \mathcal{O}_{K_{32}}[[q]]
\end{eqnarray*}
give rise to Hecke eigenforms in 
$\mathscr{S}_{13/2}(\Gamma_0(4))^{(1)}$ and 
$\mathscr{S}_{33/2}(\Gamma_0(4))^{(1)}$ 
corresponding respectively to $f_{12}$ and $f_{32}$ 
via the Shimura correspondence, 
where $\mathcal{O}_{K_{32}}$ denotes the ring of integers in the real quadratic field 
$K_{32}$. 
By making use of the induction formulas of $F_l(T;\,X)$ (cf. \cite{Kat99}), 
we computed Fourier coefficients of ${\rm Lift}^{(4)}(f_{12}) \in \mathscr{S}_8({\rm Sp}_4(\mathbb{Z}))$ and ${\rm Lift}^{(4)}(f_{32}) \in \mathscr{S}_{18}({\rm Sp}_4(\mathbb{Z}))$ for 
$4475$ half-integral symmetric matrices in 
$
{\rm Sym}_4^{*}(\mathbb{Z})$  
according to Nipp's table of equivalent classes of quaternary quadratic forms with discriminant up to $457$ (cf. \cite{Nipp}). 
For simplicity, 
we denote by $$[t_{11}, t_{22}, t_{33}, t_{44}, t_{12}, t_{13}, t_{23}, t_{14}, t_{24}, t_{34}]$$ 
the 
matrix 
\[
\left(
\begin{array}{cccc}
  t_{11} & {t_{12}}/2 & {t_{13}}/2 & {t_{14}}/2 \\
  {t_{12}}/2 & t_{22} & {t_{23}}/2 & {t_{24}}/2 \\
  {t_{13}}/2 & {t_{23}}/2 & t_{33} & {t_{34}}/2 \\
  {t_{14}}/2 & {t_{24}}/2 & {t_{34}}/2 & t_{44} \\
\end{array}
\right) \in {\rm Sym}_4^{*}(\mathbb{Z}). 
\]
For each $0<T\in 
{\rm Sym}_4^{*}(\mathbb{Z})$ with $\mathfrak{D}_T=\det(2T)=121$, 
$\mathfrak{d}_T=1$ and $\mathfrak{f}_T=11$, 
is equivalent to one of the following three representatives: 
\[
T_1=[1,1,3,3,0,1,0,0,1,0], \hspace*{5mm} 
T_2=[1,1,4,4,1,1,0,1,1,4], 
\]
\[
T_3=[2,2,2,2,2,1,0,1,1,2].  
\]
By virtue of Theorem in \cite{Kat99}, 
we obtain 
\[
F_{11}(T_i;X)=1-1452 X+161051 X^2
\]
for $i=1,\,2,\,3$. Then we checked that 
the norm of the difference 
\begin{center}
$A_{T_i}({\rm Lift}^{(4)}(f_{12})) - A_{T_i}({\rm Lift}^{(4)}(f_{32}))$ \hspace*{5mm}$(i=1,\,2,\,3)$ 
\end{center}
is factored 
into 
\[
2^8 \cdot 3^4 \cdot 5^5 \cdot 11 \cdot 171449 \cdot 680531 \cdot 35058959130397. 
\]
Moreover, for each $0<T\in 
{\rm Sym}_4^{*}(\mathbb{Z})$ with $4 \le \mathfrak{D}_T \le 457$ appearing in \cite{Nipp}, 
the norm of $A_{T}({\rm Lift}^{(4)}(f_{12})) - A_{T}({\rm Lift}^{(4)}(f_{32}))$ is indeed divisible by $p=11$. 
This fact supports that ${\rm Lift}^{(4)}(f_{12})$ is congruent to ${\rm Lift}^{(4)}(f_{32})$ modulo a prime ideal of $\mathcal{O}_{K_{32}}$ 
lying over $p=11$, and hence their semi-ordinary $11$-stabilizations 
${\rm Lift}^{(4)}(f_{12})^{*}$ and ${\rm Lift}^{(4)}(f_{32})^{*}$ 
reside both in the same $11$-adic analytic family. Similarly, we also produced another examples of 
congruences between Fourier coefficients of ${\rm Lift}^{(4)}(f_{20})\in \mathscr{S}_{12}({\rm Sp}_4(\mathbb{Z}))$ and ${\rm Lift}^{(4)}(f_{56})\in \mathscr{S}_{30}({\rm Sp}_4(\mathbb{Z}))$
for $p=19$. 
\end{ex}

%


\end{document}